\documentclass[12pt,reqno]{article}
\usepackage{amsthm,amsmath,amssymb,enumerate,cite,bm}
\usepackage[a4paper,margin=25mm,bottom=35mm]{geometry}
\usepackage[colorlinks]{hyperref}
\usepackage[english]{babel}
\usepackage[utf8]{inputenc}
\usepackage[T1]{fontenc}

\newtheorem{thm}{Theorem}[section]
\newtheorem{lem}[thm]{Lemma}
\newtheorem{cor}[thm]{Corollary}
\newtheorem{prop}[thm]{Proposition}
\theoremstyle{definition}
\newtheorem{exa}[thm]{Example}
\theoremstyle{remark}
\newtheorem*{rem}{Remark}

\numberwithin{equation}{section}

\newcommand{\N}{\mathbb{N}}
\newcommand{\ra}{\rightarrow}
\newcommand{\rs}{\rightsquigarrow}
\newcommand{\cd}{\cdot}
\newcommand{\ci}{\cdot}
\newcommand{\st}{\star}

\renewcommand{\phi}{\varphi}
\renewcommand{\rho}{\varrho}

\newcommand{\alg}[1]{\mathbf{#1}}
\newcommand{\var}[1]{\mathcal{#1}}
\newcommand{\con}[1]{\mathrm{Con}(\mathbf{#1})}
\newcommand{\bcon}[1]{\mathbf{Con}(\mathbf{#1})}
\newcommand{\conq}[2]{\mathrm{Con}_{\mathcal{#1}}(\mathbf{#2})}
\newcommand{\bconq}[2]{\mathbf{Con}_{\mathcal{#1}}(\mathbf{#2})}

\newcommand{\bidq}[2]{\mathbf{Id}_{\mathcal{#1}}(\mathbf{#2})}
\newcommand{\fil}[1]{\mathrm{Fi}(\mathbf{#1})}
\newcommand{\bfil}[1]{\mathbf{Fi}(\mathbf{#1})}
\newcommand{\pfil}[1]{\mathrm{Pr}(\mathbf{#1})}
\newcommand{\bpfil}[1]{\mathbf{Pr}(\mathbf{#1})}

\title{Some properties of pseudo-BCK- and pseudo-BCI-algebras}

\author{Petr Emanovsk\'{y}\footnote{Corresponding author. E-mail: petr.emanovsky@upol.cz} \and Jan K\"{u}hr\footnote{E-mail: jan.kuhr@upol.cz}}

\date{\small
Department of Algebra and Geometry,
Faculty of Science, Palack\'{y} University in~Olomouc\\ 
17.~listopadu 12, 77146 Olomouc, Czech Republic}

\begin{document}
%%%%%%%%%%%%%%%%%%%%%%%%%%%%%%%%%%%
%P. Emanovsk\'{y} and J. K\"{u}hr:
%\textit{Some properties of pseudo-BCK- and pseudo-BCI-algebras}. Fuzzy Sets Syst. 339 (2018) 1--16.
%%%%%%%%%%%%%%%%%%%%%%%%%%%%%%%%%%%

\noindent
This is an accepted manuscript of an article published in Fuzzy Sets and Systems.\\ 
The final publication is available at 
\url{https://doi.org/10.1016/j.fss.2016.12.014}.

\maketitle

\begin{abstract}
Pseudo-BCI-algebras generalize both BCI-algebras and pseudo-BCK-algebras, which are a non-commutative generalization of BCK-algebras. In this paper, following \cite{RvA}, we show that pseudo-BCI-algebras are the residuation subreducts of semi-integral residuated po-monoids and characterize those pseudo-BCI-algebras which are direct products of pseudo-BCK-algebras and groups (regarded as pseudo-BCI-algebras).
We also show that the quasivariety of pseudo-BCI-algebras is relatively congruence modular; in fact, we prove that this holds true for all relatively point regular quasivarieties which are relatively ideal determined, in the sense that the kernels of relative congruences can be described by means of ideal terms.

\medskip

\noindent\emph{Keywords}: Pseudo-BCK-algebra; pseudo-BCI-algebra; filter; prefilter; ideal term; relative congruence modularity. 
\end{abstract}

%%%%%%%%%%%%%%%%%%%%%%%%%%%%%%%%%%%%%%%%%%%%%%%%%%%%%%%%%%%%%%%%%%%%%%

\section{Introduction}

BCK- and BCI-algebras were introduced by Is\'{e}ki \cite{I66} as algebras induced by Meredith's implicational logics \textsf{BCK} and \textsf{BCI}. As a matter of fact, BCK-algebras are the equivalent algebraic semantics for the logic \textsf{BCK}, but \textsf{BCI} is not algebraizable in the sense of \cite{BP89}. Nevertheless, both BCK- and BCI-algebras are closely related to residuated commutative po-monoids because, as is well-known, BCK-algebras are the $\{\ra,1\}$-subreducts of integral residuated commutative po-monoids (see \cite{Fl,OK,Pa}), and BCI-algebras are the $\{\ra,1\}$-subreducts of semi-integral ones, where ``semi-integral'' means that the monoid identity $1$ is a maximal element of the underlying poset (see \cite{RvA}).

Also some other algebras of logic, such as MV-algebras, BL-algebras or hoops, are (equivalent to) certain integral residuated commutative po-monoids. Although non-commutative residuated po-monoids were known since the 1930s (for historical overview see \cite{GJKO}), non-commutative versions of the aforementioned algebras were introduced only about 15 years ago in a series of papers by Georgescu, Iorgulescu and coauthors in which they defined pseudo-MV-algebras\footnote{Rach\r{u}nek \cite{Ra} independently introduced another non-commutative generalization of MV-algebras, called GMV-algebras. In fact, pseudo-MV-algebras and these GMV-algebras are equivalent. In \cite{GJKO}, the name ``GMV-algebra'' has a different, more general meaning.} \cite{GI-mv}, pseudo-BL-algebras \cite{DNGI-bl}, pseudo-hoops \cite{GIP-ph}, and pseudo-BCK-algebras \cite{GI}. The prefix ``pseudo'' is meant to indicate non-commutativity, so that a commutative pseudo-X-algebra is an X-algebra. We refer the reader to \cite{ciungu} or \cite{iorgulescu} for an overview of non-commutative generalizations of algebras of logic.

Pseudo-BCK-algebras are algebras $(A,\ra,\rs,1)$ with two binary operations $\ra,\rs$ and a constant $1$ such that $(A,\ra,1)$ is a BCK-algebra whenever $\ra$ and $\rs$ coincide; the exact definition is given in Section~\ref{sect-uvod}. An important ingredient is the partial order which is given by $x\leq y$ iff $x\ra y=1$ iff $x\rs y=1$. The underlying poset $(A,\leq)$ has no special properties except that $1$ is its greatest element.
A particular case of pseudo-BCK-algebras are Bosbach's cone algebras \cite{Bo} which may be thought of as the $\{\ra,\rs,1\}$-subreducts of negative cones of lattice-ordered groups.
That pseudo-BCK-algebras are indeed the $\{\ra,\rs,1\}$-subreducts of integral residuated po-monoids  was proved in \cite{JK-ja07} and independently in \cite{vA06} where pseudo-BCK-algebras are called biresiduation algebras.

In this context we should mention Komori's BCC-algebras \cite{Ko} which are the $\{\ra,1\}$-subreducts of integral one-sided residuated po-monoids (see \cite{OK}). It was noted in \cite{ZL} that an algebra $(A,\ra,\rs,1)$ is a pseudo-BCK-algebra if and only if both $(A,\ra,1)$ and $(A,\rs,1)$ are BCC-algebras satisfying the identity $x\ra (y\rs z)=y\rs (x\ra z)$.

Pseudo-BCI-algebras were introduced by Dudek and Jun \cite{DJ} as a non-commutative generalization of BCI-algebras.
They are also algebras $(A,\ra,\rs,1)$ with the same partial order as pseudo-BCK-algebras such that $(A,\ra,1)$ is a BCI-algebra when $\ra$ and $\rs$ coincide, but in contrast to pseudo-BCK-algebras, $1$ is merely a maximal element of the underlying poset. Thus, roughly speaking, the difference between pseudo-BCK- and pseudo-BCI-algebras is the same as the difference between BCK- and BCI-algebras.
In the present paper, we focus on some properties of BCI-algebras and try to establish analogous results in the setting of pseudo-BCI-algebras.
Basically, there are three topics we deal with: 
(1)~embedding of pseudo-BCI-algebras into residuated po-monoids;
(2)~congruence properties of the quasivarieties of pseudo-BCK- and pseudo-BCI-algebras; and
(3)~direct products of pseudo-BCK-algebras and groups which are regarded as pseudo-BCI-algebras.

The paper is organized as follows. Section~\ref{sect-uvod} is an introduction to residuated po-monoids, pseudo-BCK- and pseudo-BCI-algebras.
In Section~\ref{sect-vnoreni}, we prove that up to isomorphism every pseudo-BCI-algebra is a subalgebra of the $\{\ra,\rs,1\}$-reduct of a semi-integral residuated po-monoid. Actually, we only sketch the construction which is almost identical to the one for BCI-algebras and semi-integral residuated commutative po-monoids presented by Raftery and van Alten \cite{RvA}.

Section~\ref{sect-filters} is devoted to congruence properties of pseudo-BCK- and pseudo-BCI-al\-ge\-bras. They form relatively $1$-regular quasivarieties $\var{PBCK}$ and $\var{PBCI}$; the former is relatively congruence distributive, the latter only relatively congruence modular. In fact, most results are proved for any relatively ideal determined quasivariety $\var Q$, by which we mean that the kernels ($1$-classes) of relative congruences of algebras in $\var Q$ may be characterized via ideal terms, similarly to ideal determined varieties (see \cite{GU,BR99}).
We prove that the relative congruence lattices of algebras in such a quasivariety $\var Q$ are arguesian. 
We also briefly discuss prefilters of pseudo-BCK- and pseudo-BCI-algebras; the name ``filter'' is reserved for the congruence kernels. Here, we can see an analogy with subgroups of groups: when a group is regarded as a pseudo-BCI-algebra, then its subgroups are precisely the prefilters of this pseudo-BCI-algebra. 
The study of prefilter lattices has proven to be useful for pseudo-BCK-algebras as well as various classes of residuated po-monoids; see e.g. \cite{JK-ja07,JK-csmj05,DK,HK,vA02,BKLT}.

In Section~\ref{sect-direktnisoucin}, following \cite{RvA}, we characterize those pseudo-BCI-algebras which are isomorphic to the direct product of a pseudo-BCK-algebra and a group which is regarded as a pseudo-BCI-algebra.

%%%%%%%%%%%%%%%%%%%%%%%%%%%%%%%%%%%%%%%%%%%%%%%%%%%%%%%%%%%%%%%%%%%%%%

\section{Preliminaries}
\label{sect-uvod}

The basic concept is that of a residuated partially ordered monoid (po-monoid for short).
A \emph{residuated po-monoid}\footnote{It follows from \eqref{rl} that $(M,\leq,\cdot,1)$ is indeed a \emph{partially ordered monoid}, in the sense that $\leq$ is preserved by multiplication on both sides.} is a structure $\alg M = (M,\leq,\cdot,\ra,\rs,1)$, where $(M,\cdot,1)$ is a monoid, $\leq$ is a partial order\footnote{If the poset $(M,\leq)$ is a lattice, then $\alg M$---or more exactly, $(M,\vee,\wedge,\cdot,\ra,\rs,1)$ where $\vee,\wedge$ are the associated lattice operations---is called a \emph{residuated $\ell$-monoid} or a \emph{residuated lattice}.} on $M$, and $\ra,\rs$ are binary operations on $M$ satisfying the \emph{residuation law}, for all $x,y,z\in M$:
\begin{equation}
\label{rl}
x\leq y\ra z \quad\text{iff}\quad x\cd y\leq z,\quad\text{and}\quad x\leq y\rs z \quad\text{iff}\quad y\cd x\leq z.
\end{equation}
It is clear that the monoid reduct $(M,\cd,1)$ is commutative if and only if $x\ra y=x\rs y$ for all $x,y\in M$, in which case $(M,\leq,\cd,\ra,1)$ is a \emph{residuated commutative po-monoid}.
A residuated po-monoid is said to be 
\begin{enumerate}[\indent\upshape (1)]
\item \emph{semi-integral} if the monoid identity $1$ is a maximal element of the poset $(M,\leq)$;
\item \emph{integral} if $1$ is the greatest element of $(M,\leq)$.
\end{enumerate}
In the (semi-) integral case, $\leq$ is specified by either of the ``arrows'' because by \eqref{rl} one has
\begin{equation}\label{usporadani}
x\leq y \quad\text{iff}\quad x\ra y=1 \quad\text{iff}\quad x\rs y=1.
\end{equation}
Thus (semi-) integral residuated po-monoids may be likewise defined as algebras $\alg M = (M,\cdot,\ra,\rs,1)$ of type $(2,2,2,0)$. The class of such algebras is a quasivariety, but not a variety. 
The quasivariety of semi-integral residuated po-monoids can be axiomatized by
\begin{gather}
\label{rpom1a} 
(x\ra y)\rs ((y\ra z)\rs (x\ra z)) = 1,\\ 
\label{rpom1b}
(x\rs y)\ra ((y\rs z)\ra (x\rs z)) = 1,\\
\label{rpom2a}
1\ra x = x,\\
\label{rpom2b}
1\rs x = x,\\
\label{rpom3}
(x\cd y)\ra z = x\ra (y\ra z),\\
\label{rpom4}
x\ra y = 1 \quad\&\quad y\ra x = 1 \quad\Rightarrow\quad x = y,
\end{gather}
and the quasivariety of integral residuated po-monoids by \eqref{rpom1a}--\eqref{rpom4} together with
\begin{equation}\label{integral}
x\ra 1 = 1.
\end{equation}
The proof is straightforward. In \eqref{rpom4} and \eqref{integral} we could equally use $\rs$ instead of $\ra$, and \eqref{rpom3} could be replaced by the identity $(x\cd y)\rs z = y\rs (x\rs z)$.

Pseudo-BCK-algebras were introduced by Georgescu and Iorgulescu \cite{GI} as a non-com\-mu\-ta\-tive generalization of BCK-algebras; they are equivalent to biresiduation algebras which were defined by van Alten \cite{vA06}.
Later, pseudo-BCI-algebras were defined by Dudek and Jun \cite{DJ} along the same lines, thus they generalize both BCI-algebras and pseudo-BCK-algebras. There are several possible definitions out of which we choose this (cf. \cite{DKD13} or \cite{JK-hab}):

A \emph{pseudo-BCI-algebra} is an algebra $\alg{A}=(A,\ra,\rs,1)$ of type $(2,2,0)$ satisfying the identities \eqref{rpom1a}--\eqref{rpom2b} and the quasi-identity \eqref{rpom4},
and a \emph{pseudo-BCK-algebra} is a pseudo-BCI-algebra which in addition satisfies the identity \eqref{integral}.\footnote{As in the case of residuated po-monoids, in \eqref{rpom4} and \eqref{integral}, $\ra$ could be equivalently replaced by $\rs$.}
If $\alg A$ is a pseudo-BCI-algebra (resp. pseudo-BCK-algebra) in which $x\ra y=x\rs y$ for all $x,y\in A$, then the algebra $(A,\ra,1)$ is a \emph{BCI-algebra} (resp. \emph{BCK-algebra}).\footnote{Is\'{e}ki \cite{I66} originally defined BCI- and BCK-algebras as algebras $(A,*,0)$, where $*$ may be thought of as a kind of subtraction and $0$ is a minimal (or the least) element of the underlying poset $(A,\leq)$ which is defined by $x\leq y$ iff $x*y=0$. 
Similarly, in \cite{GI} and \cite{DJ}, pseudo-BCI- and pseudo-BCK-algebras were defined as algebras $(A,*,\circ,0)$ with two subtractions $*,\circ$ and minimal element $0$.}

For any pseudo-BCI-algebra $\alg{A}$, the relation $\leq$ given by \eqref{usporadani}---more precisely, by $x\leq y$ iff $x\ra y=1$, which is equivalent to $x\rs y=1$---is a partial order on $A$. The poset $(A,\leq)$ is the \emph{underlying poset} of $\alg A$. The difference between pseudo-BCK- and pseudo-BCI-algebras is that in the former case, $1$ is the greatest element of the underlying poset, while in the latter case, $1$ is only a maximal element. Thus, pseudo-BCK-algebras are integral and pseudo-BCI-algebras semi-integral.

Of course, pseudo-BCI- and pseudo-BCK-algebras are closely related to residuated po-monoids. If $\alg{M}=(M,\cdot,\ra,\rs,1)$ is a semi-integral (resp. integral) residuated po-monoid, then every subalgebra of the reduct $(M,\ra,\rs,1)$ is a pseudo-BCI-algebra (resp. pseudo-BCK-algebra). 
All pseudo-BCK-algebras (biresiduation algebras) arise in this way, i.e., up to isomorphism, every pseudo-BCK-algebra is a $\{\ra,\rs,1\}$-subreduct of an integral residuated po-monoid. In Section \ref{sect-vnoreni} we establish an analogue for pseudo-BCI-algebras and semi-integral residuated po-monoids.

An easy but important observation is this: 
Given a pseudo-BCI-algebra (resp. pseudo-BCK-algebra) $\alg{A}=(A,\ra,\rs,1)$, the algebra $\alg{A}^\dagger=(A,\rs,\ra,1)$ is a pseudo-BCI-algebra (resp. pseudo-BCK-algebra), too. It is plain that $\alg{A}$ and $\alg{A}^\dagger$ have the same underlying poset $(A,\leq)$, but it can easily happen that the algebras $\alg A$ and $\alg A^\dagger$ are \emph{not} isomorphic. For example, the prelinearity identities
\begin{gather*}
(x\ra y)\ra z\leq ((y\ra x)\ra z)\ra z \quad\text{and}\quad
(x\rs y)\rs z\leq ((y\rs x)\rs z)\rs z
\end{gather*}
are independent in general. Note that if $\alg A$ satisfies either of the two identities, then $z\leq 1$ holds in $\alg A$, so $\alg A$ is a pseudo-BCK-algebra.

In the following lemma we list the basic arithmetical properties of pseudo-BCI-al\-ge\-bras. For completeness, we also include some properties that have been implicitly mentioned above, such as \eqref{vl1}, \eqref{vl2} or \eqref{vl4}.

\begin{lem}[cf. \cite{DJ}, \cite{JKN}, \cite{LP}]
\label{L1}
In any pseudo-BCI-algebra $\alg{A}$, for all $x,y,z\in A$:
\begin{enumerate}[\indent\upshape (1)]
\item\label{vl1} $x\ra x=1$, $x\rs x=1$;
\item\label{vl2} $x\ra y\leq (y\ra z)\rs (x\ra z)$, $x\rs y\leq (y\rs z)\ra (x\rs z)$;
\item\label{vl3} $x\leq (x\ra y)\rs y$, $x\leq (x\rs y)\ra y$;
\item\label{vl4} $x\ra y=1$ iff $x\rs y=1$;
\item\label{vl5} $x\leq y$ implies $y\ra z\leq x\ra z$ and $y\rs z\leq x\rs z$;
\item\label{vl6} $x\ra (y\rs z)=y\rs (x\ra z)$;
\item\label{vl7} $x\leq y\ra z$ iff $y\leq x\rs z$;
\item\label{vl8} $x\ra y\leq (z\ra x)\ra (z\ra y)$, $x\rs y\leq (z\rs x)\rs (z\rs y)$;
\item\label{vl9} $x\leq y$ implies $z\ra x\leq z\ra y$ and $z\rs x\leq z\rs y$;
\item\label{vl10} $x\ra 1=x\rs 1$;
\item\label{vl11} $(x\ra y)\ra 1=(x\ra 1)\rs (y\ra 1)$, $(x\rs y)\rs 1=(x\rs 1)\ra (y\rs 1)$;
\item $((x\ra y)\rs y)\ra y=x\ra y$, $((x\rs y)\ra y)\rs y=x\rs y$.
\end{enumerate}
\end{lem}

Note that the inequalities $y\leq x\ra y$ and $y\leq x\rs y$ hold true only in pseudo-BCK-algebras
(because $y\leq x\ra y$ is equivalent to $x\leq y\rs y=1$).

%%%%%%%%%%%%%%%%%%%%%%%%%%%%%%%%%%%%%%%%%%%%%%%%%%%%%%%%%%%%%%%%%%%%%%

Let $\alg A$ be a pseudo-BCI-algebra.
The identities \eqref{vl11} of Lemma \ref{L1} entail that the set
\[
I_{\alg{A}} = \{x\in A \mid x\leq1\}, 
\]
which we call the \emph{integral part} of $\alg{A}$, is a subuniverse of $\alg{A}$. Clearly, the subalgebra $\alg{I}_{\alg{A}}=(I_{\alg A},\ra,$ $\rs,1)$ is a pseudo-BCK-algebra; in fact, $\alg{I}_{\alg{A}}$ is the largest subalgebra of $\alg A$ which is a pseudo-BCK-algebra. 
By the \emph{group part} of the pseudo-BCI-algebra $\alg{A}$ we mean the set 
\[
G_{\alg{A}} = \{x\ra 1 \mid x\in A\} = \{x\rs 1 \mid x\in A\}.
\]
Again, by \eqref{vl10} and \eqref{vl11}, $G_{\alg A}$ is a subuniverse of $\alg{A}$. The name ``group part'' is justified by the fact that the subalgebra $\alg{G}_{\alg{A}}=(G_{\alg A},\ra,\rs,1)$ is term equivalent to a group. Indeed, if we put 
\[
g\cdot h = (g\ra 1) \rs h
\]
for $g,h\in G_{\alg{A}}$, then $(G_{\alg{A}},\cdot,1)$ is a group in which 
$g^{-1} = g\ra 1 = g\rs 1$ is the inverse of $g\in G_{\alg{A}}$; the original operations $\ra$ and $\rs$ on $G_{\alg{A}}$ are retrieved by 
\[
g\ra h=h\cdot g^{-1} \quad\text{and}\quad g\rs h=g^{-1}\cdot h,
\]
respectively. This was essentially proved in \cite{D12a,D12b} (also see \cite{Z10a,Z10b}).

If we define $g*h=h\cdot g$, then $(G_{\alg{A}},*,1)$ is a group which is isomorphic to $(G_{\alg{A}},\cdot,1)$ and gives rise to the pseudo-BCI-algebra $\alg{G}_{\alg{A}}^{\dagger}=(G_{\alg{A}},\rs,\ra,1)$, because $g\rs h=g^{-1}\cdot h=h*g^{-1}$ and $g\ra h=h\cdot g^{-1}=g^{-1}*h$.
Consequently, the map $g\in G_{\alg{A}}\mapsto g^{-1}\in G_{\alg{A}}$ is an isomorphism between $\alg{G}_{\alg{A}}$ and $\alg{G}_{\alg{A}}^{\dagger}$ (though $\alg A$ and $\alg A^\dagger$ need not be isomorphic).

The following was independently proved in several papers; see \cite{D12a,D12b}, \cite{KS} or \cite{Z10a,Z10b}. The proof is easy anyway.

\begin{lem}\label{lem-grcast}
For any pseudo-BCI-algebra $\alg A$, $G_{\alg{A}}$ is the set of the maximal elements of $(A,\leq)$. 
For every $x\in A$, the element $g=(x\ra 1)\ra 1$ is the only $g\in G_{\alg{A}}$ such that $x\leq g$. Therefore, $x\in G_{\alg{A}}$ iff $x=(x\ra 1)\ra 1$. 
\end{lem}

Analogously, for the integral part we have:

\begin{lem}\label{lem-intcast}
Let $\alg A$ be a pseudo-BCI-algebra. For any $x\in A$, $((x\ra 1)\ra 1)\ra x\in I_{\alg A}$, and hence $x\in I_{\alg A}$ iff $x=((x\ra 1)\ra 1)\ra x$.
\end{lem}

\begin{proof}
Since $x\leq (x\ra 1)\ra 1$, it follows $1=x\ra x\geq ((x\ra 1)\ra 1)\ra x$.
If $x\leq 1$, then $((x\ra 1)\ra 1)\ra x=(1\ra 1)\ra x=1\ra x=x$.
\end{proof}

We adopt the notation of \cite{RvA}, but the sets $I_{\alg A}$ and $G_{\alg A}$ can be found in the literature under various names and symbols. 
$I_{\alg A}$ is often called the \emph{pseudo-BCK-part} of $\alg A$ and denoted by $K(\alg A)$, and $G_{\alg A}$ is called the \emph{anti-grouped part} or the \emph{center} of $\alg A$, denoted by $AG(\alg A)$ or $At(\alg A)$, where ``$At$'' comes from ``atoms''. Pseudo-BCI-algebras satisfying $G_{\alg A}=A$ are called \emph{p-semisimple}.

Another important consequence of Lemma \ref{L1} \eqref{vl10} and \eqref{vl11} is that the map 
\begin{equation}\label{gamma}
\gamma\colon x\in A\mapsto x\ra 1\in G_{\alg{A}}
\end{equation}
is a homomorphism from $\alg{A}$ onto $\alg{G}_{\alg{A}}^\dagger$, and 
\begin{equation}\label{delta}
\delta\colon x\in A\mapsto (x\ra 1)\ra 1\in G_{\alg{A}}
\end{equation}
is a homomorphism from $\alg{A}$ onto $\alg{G}_{\alg{A}}$.
We have $\alg{A}/\ker\gamma\cong\alg{G}_{\alg{A}}^\dagger \cong\alg{G}_{\alg{A}} \cong \alg{A}/\ker\delta$ and $\gamma^{-1}(1)=\delta^{-1}(1)=I_{\alg{A}}$, thus $I_{\alg A}$ is always the kernel of a relative congruence of $\alg A$, while $G_{\alg A}$ is not in general (see Section~\ref{sect-filters}).

Given any pseudo-BCK-algebra $\alg B$ and any group $\alg H$ (regarded as a pseudo-BCI-al\-ge\-bra), we can always construct a pseudo-BCI-algebra $\alg A$ such that $\alg I_{\alg A}\cong \alg B$ and $\alg G_{\alg A}\cong \alg H$.
For example, we can take the direct product $\alg A=\alg B\times\alg H$ because $I_{\alg A}=\{(b,1)\mid b\in B\}$ and $G_{\alg A}=\{(1,h)\mid h\in H\}$.
In Section~\ref{sect-direktnisoucin} we characterize those pseudo-BCI-algebras $\alg{A}$ which are isomorphic to the direct product $\alg{I}_{\alg{A}}\times\alg{G}_{\alg{A}} \cong \alg{I}_{\alg{A}}\times\alg{G}_{\alg{A}}^\dagger$.
But there is an easier construction, see \cite{Z13} (in fact, Theorem 4.3 therein presents a construction of a semi-integral residuated po-monoid from a bounded integral residuated po-monoid and a group).

Let $\alg{B}=(B,\ra,\rs,1)$ be a pseudo-BCK-algebra, $(H,\cdot,1)$ be a group such that $B\cap H=\{1\}$, and let $\alg{H}=(H,\ra,\rs,1)$ be the pseudo-BCI-algebra derived from the group, i.e. $g\ra h=h\cdot g^{-1}$ and $g\rs h=g^{-1}\cdot h$.
Let $A=B\cup H$ be equipped with the operations $\ra,\rs$ defined as follows:
\[
x\diamond y=
\begin{cases}
x\diamond y & \text{if } x,y\in B \text{ or } x,y\in H,\\
y & \text{if } x\in B, y\in H,\\
x^{-1} & \text{if } x\in H\setminus\{1\}, y\in B,
\end{cases}
\]
where $\diamond\in\{\ra,\rs\}$.
Then $\alg{A}=(A,\ra,\rs,1)$ is a pseudo-BCI-algebra with $\alg I_{\alg{A}}=\alg B$ and $\alg G_{\alg{A}}=\alg H$. The proof is but a direct inspection of all possible cases.
Note that $\alg{A}$ is not a subalgebra of the direct product $\alg{B}\times\alg{H}$. 

%%%%%%%%%%%%%%%%%%%%%%%%%%%%%%%%%%%%%%%%%%%%%%%%%%%%%%%%%%%%%%%%%%%%%%

\section{Embedding into the $\{\ra,\rs,1\}$-reducts of residuated po-monoids}
\label{sect-vnoreni}

It is folklore that BCK-algebras are the $\{\ra,1\}$-subreducts of integral residuated commutative po-monoids (see \cite{Fl,OK,Pa}). Similarly, pseudo-BCK-algebras (biresiduation algebras) are the $\{\ra,\rs,1\}$-subreducts of integral residuated po-monoids, which was proved independently by van Alten in \cite{vA06} and by the second author in \cite{JK-cga}. Also see \cite[Theorem 1.2.1]{JK-hab}, the proof is just a ``non-commutative modification'' of the construction of \cite{OK}. In \cite[Theorem 2]{RvA}, Raftery and van Alten proved that BCI-algebras are the $\{\ra,1\}$-subreducts of semi-integral residuated commutative po-monoids, and the aim of this section is to generalize their result to pseudo-BCI-algebras.

In proving that a pseudo-BCK-algebra $\alg{A}$ is isomorphic to a subalgebra of the residuation reduct of an integral residuated po-monoid, one first constructs a certain po-monoid, say $\alg{J}(\alg{A})$, whose identity $\{1\}$ is also its smallest element, and then takes the set of order-filters of this intermediate po-monoid in order to obtain the integral residuated po-monoid into which $\alg{A}$ embeds.
In the case that $\alg{A}$ is a pseudo-BCI-algebra, the first step remains unchanged, but $\{1\}$ is merely a minimal element in the po-monoid $\alg{J}(\alg{A})$, and hence the second step must be modified as in \cite{RvA}.

Let $\alg{M}=(M,\leq,*,e)$ be a po-monoid in which $e$ is a minimal element. 
Let $G_{\alg{M}}$ be the set of the minimal elements of $(M,\leq)$ and suppose that every $a\in M$ exceeds a unique $g\in G_{\alg{M}}$. Suppose further that $(G_{\alg{M}},\cdot,e)$ is a group such that $g\cdot h\leq g*h$ for all $g,h\in G_{\alg{M}}$. 
Note that we do not require that $(G_{\alg{M}},\cdot,e)$ be a subgroup of $(M,*,e)$.

Let $\mathfrak{F}_{\alg{M}}$ be the set of all (non-empty) order-filters $X$ in the poset $(M,\leq)$ with the property that $X\subseteq [g)$\footnote{As usual, for an element $p$ in a poset $(P,\leq)$, $[p)$ denotes the order-filter $\{x\in P\mid x\geq p\}$.} for some $g\in G_{\alg{M}}$. 
Note that $[a)\in \mathfrak{F}_{\alg{M}}$ for every $a\in M$, because there is a unique $g\in G_{\alg{M}}$ with $g\leq a$, and hence the map $a\mapsto [a)$ is an antitone injection from $(M,\leq)$ into $(\mathfrak{F}_{\alg{M}},\subseteq)$.
For any $X,Y\in \mathfrak{F}_{\alg{M}}$, let
\begin{gather*}
X\odot Y = \{a\in M\mid a\geq x*y \text{ for some } x\in X,y\in Y\},\\
X\ra Y = \{a\in M\mid [a)\odot X\subseteq Y\}, \text{ and}\\
X\rs Y = \{a\in M\mid X\odot [a)\subseteq Y\}.
\end{gather*}

\begin{lem}\label{vnoreni1}
If $\alg{M}=(M,\leq,*,e)$ is a po-monoid satisfying the above conditions, then
$\bm{\mathfrak{F}}_{\alg{M}}=(\mathfrak{F}_{\alg{M}},\subseteq,$ $\odot,\ra,\rs,[e))$ is a semi-integral residuated po-monoid.
\end{lem}

\begin{proof}
As in \cite{RvA}, we leave the details to the reader, we only make two remarks.
First, if $X\subseteq [g)$ and $Y\subseteq [h)$ where $g,h\in G_{\alg{M}}$, then $x*y\geq g*h\geq g\cdot h\in G_{\alg{M}}$ for all $x\in X$ and $y\in Y$, whence $X\odot Y\in\mathfrak{F}_{\alg{M}}$.
Second, if $a\in X\ra Y$ where $a\geq f\in G_{\alg{M}}$, then $[a)\odot X\subseteq Y$ and so $a*x\geq h$ for any $x\in X$, and at the same time, $a*x\geq f\cdot g$. Since there is a unique minimal element below $a*x$, it follows that $h=f\cdot g$, whence $a\geq f=h\cdot g^{-1}\in G_{\alg{M}}$. This shows that $X\ra Y\subseteq [h\cdot g^{-1})$, and so $X\ra Y\in\mathfrak{F}_{\alg{M}}$. Analogously, $X\rs Y\in\mathfrak{F}_{\alg{M}}$.
\end{proof}

Now, let $\alg{A}=(A,\ra,\rs,1)$ be a pseudo-BCI-algebra. As in Section \ref{sect-uvod}, $G_{\alg{A}}$ denotes the group part of $\alg{A}$. 
Let $W_A$ be the set of all non-empty words $\alpha=a_1\dots a_n$ over the set $A$. For any such $\alpha\in W_A$ and $x\in A$ we write
\begin{align*}
\alpha \ra x &= a_1\ra (\dots \ra (a_n\ra x)\dots),\\
\alpha \rs x &= a_n\rs (\dots \rs (a_1\rs x)\dots).
\end{align*}
In view of Lemma \ref{L1} \eqref{vl4} and \eqref{vl6}, we have
$\alpha \ra x=1$ iff $\alpha \rs x=1$, and also $\alpha \ra (\beta \rs x) = \beta \rs (\alpha \ra x)$ for all $\alpha,\beta\in W_A$.
For any $\alpha\in W_A$, let
\[
J(\alpha) = \{x\in A\mid \alpha \ra x=1\},
\]
and let $J(\alg{A})=\{J(\alpha)\mid \alpha\in W_A\}$.
Note that $J(a)=[a)$ for any $a\in A$, and hence $J(g)=\{g\}$ for any $g\in G_{\alg{A}}$.
Also, $a\leq b$ iff $J(b)\subseteq J(a)$ for all $a,b\in A$, whence the map $a\mapsto J(a)$ is an antitone injection from $(A,\leq)$ into $(J(\alg{A}),\subseteq)$.
Further, let 
\[
J(\alpha)*J(\beta)=J(\alpha\beta);
\]
thus $J(\alpha)*J(\beta)=\{x\in A\mid \alpha\ra (\beta\ra x)=1\}$.

\begin{lem}\label{vnoreni2}
For any pseudo-BCI-algebra $\alg{A}$, $\alg{J}(\alg{A}) = (J(\alg{A}),\subseteq,*,\{1\})$ is a po-monoid such that $J(g)=\{g\}$ with $g\in G_{\alg{A}}$ are the minimal elements of $(J(\alg{A}),\subseteq)$.
Every $J(\alpha) \in J(\alg{A})$ contains a unique $g\in G_{\alg{A}}$, namely, $g=(\alpha\ra 1)\ra 1$. 
Moreover, for all $\alpha,\beta\in W_A$ and $g,h\in G_{\alg{A}}$, if $g\in J(\alpha)$ and $h\in J(\beta)$, then $g\cdot h\in J(\alpha)*J(\beta)$, where $g\cdot h$ is calculated in the group $(G_{\alg A},\cdot,1)$.
\end{lem}

\begin{proof}
Again, the proof is straightforward. We only prove the last two statements.
If $g\in J(\alpha)$ where $g\in G_{\alg{A}}$, then $1=\alpha\ra g=\alpha\ra ((g\ra 1)\rs 1)=(g\ra 1)\rs (\alpha\ra 1)$, so $g\ra 1\leq \alpha\ra 1$ whence $g=(g\ra 1)\ra 1\geq (\alpha\ra 1)\ra 1$. Since $g$ and $(\alpha\ra 1)\ra 1$ are maximal elements, $g=(\alpha\ra 1)\ra 1$.
Furthermore, if $g\in J(\alpha)$ and $h\in J(\beta)$ where $g,h\in G_{\alg{A}}$, then 
$\alpha\ra (\beta\ra (g\cdot h))=\alpha\ra [\beta\ra ((g\ra 1)\rs h)]=\alpha\ra [(g\ra 1)\rs (\beta\ra h)]=\alpha\ra ((g\ra 1)\rs 1)=\alpha\ra g=1$.
\end{proof}

Since $(G_{\alg{A}},\cdot,1)$ is a group, it follows that also $(G_{\alg{J}(\alg{A})},\cd,\{1\})$ is a group, where $G_{\alg{J}(\alg{A})}=\{\{g\}\mid g\in G_{\alg{A}}\}$ and $\{g\}\cdot\{h\}=\{g\cdot h\}$.\footnote{In general, $\{g\cdot h\}\neq\{g\}*\{h\}$ for $g,h\in G_{\alg A}$ because it can happen that $g\ra (h\ra x)=1$, but $(g\cdot h)\ra x\neq 1$, or vice versa.} 
Moreover, by the last statement of Lemma \ref{vnoreni2} we have $\{g\}\cdot\{h\} \subseteq \{g\}*\{h\}$, thus the po-monoid $\alg{J}(\alg{A})=(J(\alg{A}),\subseteq,*,\{1\})$ satisfies the assumptions of Lemma \ref{vnoreni1}, and hence $\bm{\mathfrak{F}}_{\alg{J}(\alg{A})}=(\mathfrak{F}_{\alg{J}(\alg{A})},\subseteq,\odot,\ra,\rs,[\{1\}))$ is a semi-integral residuated po-monoid.
If $\alg{A}$ is a pseudo-BCK-algebra, then $\bm{\mathfrak{F}}_{\alg{J}(\alg{A})}$ is integral.

Given $x\in A$, we have $J(\alpha)\in [J(x))$ iff $J(x)\subseteq J(\alpha)$ iff $x\in J(\alpha)$, 
whence the composite injection $A\to J(\alg{A})\to \mathfrak{F}_{\alg{J}(\alg{A})}$ is given by
\[
x\mapsto J(x)\mapsto [J(x))=\{J(\alpha)\in J(\alg{A})\mid x\in J(\alpha)\}.
\]
We conclude:

\begin{thm}
For any pseudo-BCI-algebra $\alg{A}=(A,\ra,\rs,1)$, the map 
\[
x\in A \mapsto \{J(\alpha)\in J(\alg{A})\mid x\in J(\alpha)\} \in \mathfrak{F}_{\alg{J}(\alg{A})}
\]
is an embedding of $\alg{A}$ into the pseudo-BCI-algebra $(\mathfrak{F}_{\alg{J}(\alg{A})},\ra,\rs,[\{1\}))$.
Consequently, pseudo-BCI-algebras {\upshape(}resp. pseudo-BCK-algebras{\upshape)} are the $\{\ra,\rs,1\}$-subreducts of se\-mi-integral {\upshape(}resp. integral{\upshape)} residuated po-mo\-noids. 
\end{thm}

%%%%%%%%%%%%%%%%%%%%%%%%%%%%%%%%%%%%%%%%%%%%%%%%%%%%%%%%%%%%%%%%%%%%%%

\section{Relative congruences, filters and prefilters}
\label{sect-filters}

First, we recall some general facts; see e.g. \cite{BR99,BR08}.
Let $\var{Q}$ be a quasivariety of algebras of a given language with a constant $1$. A congruence $\theta$ of an algebra $\alg A\in\var Q$ is a \emph{relative congruence} (or a \emph{$\var Q$-congruence}) if the quotient algebra $\alg A/\theta$ belongs to $\var Q$. The set of all relative congruences of $\alg A$ is denoted by $\conq Q A$; ordered by set-inclusion, it forms an algebraic lattice $\bconq Q A$.
If the lattice of (relative) congruences of every algebra $\alg A\in\var Q$ is modular or distributive, then $\var Q$ is said to be (\emph{relatively}) \emph{congruence modular} or (\emph{relatively}) \emph{congruence distributive}, respectively.

Furthermore, $\var Q$ is (\emph{relatively}) \emph{$1$-regular} if distinct (relative) congruences of algebras $\alg A\in\var Q$ have distinct kernels; in other words, if $\phi,\psi$ are (relative) congruences of $\alg{A}$ with $[1]_\phi = [1]_\psi$, then $\phi=\psi$. 
It is known that $\var Q$ is relatively $1$-regular if and only if there exist binary terms $d_1,\dots,d_n$ such that $\var Q$ satisfies 
\[\textstyle
\bigwedge_{i=1}^{n} d_i(x,y) = 1 \quad\Leftrightarrow\quad x = y.
\]
In this case, every relative congruence $\theta\in\conq Q A$ is determined by its kernel $[1]_\theta$ by
\[
(a,b)\in\theta \quad\text{iff}\quad d_i(a,b)\in [1]_\theta \text{ for all } i=1,\dots,n,
\]
and the kernels of relative congruences of $\alg A$ form an algebraic lattice which is isomorphic to $\bconq{Q}{A}$; the (inverse) isomorphism is given by $\theta\mapsto [1]_\theta$.
It is also known that a $1$-regular variety is always congruence modular, but a relatively $1$-regular quasivariety need not be relatively congruence modular.

Now, let $\var{PBCK}$ and $\var{PBCI}$ be respectively the quasivariety of pseudo-BCK-al\-ge\-bras and the quasivariety of pseudo-BCI-algebras. Both $\var{PBCK}$ and $\var{PBCI}$ are relatively $1$-regular (but not $1$-regular), as witnessed by the terms $d_1(x,y)=x\diamond y$ and $d_2(x,y)=y\diamond x$ where $\diamond$ is either of the arrows $\ra$ and $\rs$. Hence every relative congruence $\theta$ is determined by its kernel $[1]_\theta$ by: 
$(a,b)\in\theta$ iff $a\diamond b,b\diamond a\in [1]_\theta$.

An internal characterization of the kernels of relative congruences of pseudo-BCI-algebras, similar to that for pseudo-BCK-algebras, can be found in \cite{D14} (see \cite{HK,JK-hab} for pseudo-BCK-algebras). Given a pseudo-BCI-algebra $\alg A$, a subset $F\subseteq A$ is the kernel $[1]_\theta$ of a (unique) relative congruence $\theta\in\conq{PBCI}{A}$ if and only if 
\begin{enumerate}[\indent (i)]
\item $1\in F$;
\item if $a,a\ra b\in F$, then $b\in F$;
\item if $a\in F$, then $a\ra 1\in F$; and
\item for all $a,b\in A$, $a\ra b\in F$ iff $a\rs b\in F$.
\end{enumerate}
In \cite{D14}, such subsets are called ``closed compatible deductive systems'', but we will call them ``filters'' (see below).
It is not hard to show that in (ii) we can use $\rs$ in place of $\ra$ and that (iv) is equivalent to the condition
\begin{enumerate}[\indent (i)]\setcounter{enumi}{4}
\item $(b\ra a)\ra a, (b\rs a)\rs a\in F$ for all $a\in A$, $b\in F$.
\end{enumerate}
Of course, (iii) is redundant if $\alg A$ is a pseudo-BCK-algebra, and (iv) or (v) is redundant if $\alg A$ is a BCI-algebra.

There are quite a few papers devoted to the subsets satisfying (i) and (ii); they are called ``deductive systems'' \cite{D14}, ``filters'' \cite{ZJ14} or ``pseudo-BCI-filters'' \cite{Z10b},
and ``ideals'' or ``pseudo-BCI-ideals'' when the dual presentation of pseudo-BCI-algebras is used \cite{D12a,JKN,LP}. The adjective ``closed'' is added when also the condition (iii), which is equivalent to being a subalgebra, is satisfied \cite{D14,D12a}.

Our terminology is hopefully simpler: Given a pseudo-BCI-algebra $\alg A$, we say that $F\subseteq A$ is
\begin{enumerate}[\indent\upshape (1)]
\item a \emph{filter} of $\alg A$ if $F$ is the kernel $[1]_\theta$ of some $\theta\in\conq{PBCI}{A}$, i.e., if $F$ satisfies the above conditions (i)--(iv), or equivalently, the conditions (i)--(iii) and (v);
\item a \emph{prefilter} of $\alg A$ if it satisfies (i)--(iii).
\end{enumerate}
Here, we require that also (iii) be satisfied because when a group is regarded as a pseudo-BCI-algebra, then the subgroups are precisely the prefilters.

We let $\fil A$ and $\pfil A$ denote respectively the set of filters and the set of prefilters of $\alg A$; it is obvious that with respect to set-inclusion they form algebraic lattices, $\bfil A$ and $\bpfil A$.
We know that $\bfil A \cong \bconq{PBCI}{A}$ under the mutually inverse maps $F\mapsto\theta_F$ and $\theta\mapsto [1]_\theta$ where $\theta_F$ is given by
\[
(a,b)\in\theta_F \quad\text{iff}\quad a\ra b,b\ra a\in F \quad\text{iff}\quad a\rs b,b\rs a\in F.
\]
Dymek \cite{D14} proved somewhat more; he proved that
if $F$ is a ``compatible deductive system'' not necessarily ``closed'' (i.e., if $F$ satisfies (i), (ii) and (iv)), then $\theta_F\in\con A$ and $F\subseteq [1]_{\theta_F}$, with equality exactly if $F$ is ``closed''.

It is known that the quasivariety of BCK-algebras $\var{BCK}$ is relatively congruence distributive, while the quasivariety of BCI-algebras $\var{BCI}$ is only relatively congruence modular (see \cite{BR99,BR08,RvA}).
It comes as no surprise that likewise $\var{PBCK}$ is relatively congruence distributive; more directly, for any pseudo-BCK-algebra $\alg A$, $\bpfil A$ is a distributive lattice and $\bfil A$ is its complete sublattice (see \cite[Corollary 2.1.12, 2.2.9]{JK-hab} and \cite{HK}).
However, $\var{PBCI}$ can be relatively congruence modular at best, because it contains a subclass which is term equivalent to groups, so the filter lattice $\bfil A$ need not be distributive.
Note that---since subgroups correspond to prefilters---the prefilter lattice $\bpfil A$ need not be even modular.

In what follows, we prove that $\var{PBCI}$ is indeed relatively congruence modular. We actually prove that this holds true for any relatively $1$-regular quasivariety which has a ``nice description'' of the kernels of relative congruences, by which we mean a description by means of the so-called ``ideal terms'' (see \cite{GU,BR99}).

As before, let $\var Q$ be a quasivariety whose language has a constant $1$. 
An \emph{ideal term}\footnote{Since the concept of an ideal term as well as that of an ideal obviously depends on the choice of $\var Q$, we should more accurately speak of $\var Q$-ideal terms and $\var Q$-ideals.} (in variables $y_1,\dots,y_n$) is a term $t(x_1,\dots,x_m,y_1,\dots,y_n)$ in the language of $\var Q$ such that $\var{Q}$ satisfies the identity
\[
t(x_1,\dots,x_m,1,\dots,1) = 1.
\]
An \emph{ideal}\footnotemark[\value{footnote}]{} of an algebra $\alg A\in\var Q$ is a non-empty subset $I$ of $A$ which is closed under all ideal terms, in the sense that for every ideal term $t(x_1,\dots,x_m,y_1,\dots,y_n)$ one has $t(a_1,\dots,a_m,b_1,\dots,b_n)\in I$ for all $a_1,\dots,a_m\in A$ and $b_1,\dots,b_n\in I$.
Ordered by set-inclusion, the ideals of $\alg A$ form an algebraic lattice $\bidq Q A$. It is easy to see that the ideal generated by $\emptyset\neq S\subseteq A$ consists precisely of the elements $t(a_1,\dots,a_m,b_1,\dots,b_n)$ where $t(x_1,\dots,x_m,y_1,\dots,y_n)$ is an ideal term and $a_1,\dots,a_m\in A$, $b_1,\dots,b_n\in S$.

$\var{Q}$ is said to be \emph{ideal determined} if for every $\alg{A}\in\var{Q}$, every ideal $I$ of $\alg{A}$ is the kernel $[1]_\theta$ of a unique congruence $\theta\in\con{A}$, in which case the map $\theta\mapsto [1]_\theta$ is an isomorphism between the lattices $\bcon A$ and $\bidq Q A$.

Analogously, we will say that $\var Q$ is \emph{relatively ideal determined} provided that every ideal $I$ of $\alg{A}$ is the kernel of a unique relative congruence $\theta\in\conq{Q}{A}$, for all $\alg{A}\in\var{Q}$. If $\var Q$ has this property, then $\bconq{Q}{A} \cong \bidq {Q}{A}$ for all $\alg{A}\in\var{Q}$.

By \cite{GU}, every ideal determined variety is congruence modular. Following the proof given in \cite{GU}, we prove that every relatively ideal determined quasivariety is relatively congruence modular.

Recalling the conditions (i)--(iii) and (v), one can easily find a finite basis of ideal terms for pseudo-BCK- and pseudo-BCI-algebras:

\begin{prop}\label{prop-idealterms}
Let $\alg A$ be a pseudo-BCI-algebra. For any $\emptyset\neq F\subseteq A$, the following are equivalent:
\begin{enumerate}[\indent\upshape (1)]
\item $F$ is a filter of $\alg A$;
\item $F$ is an ideal of $\alg A$, i.e., $F$ is closed under all $\var{PBCI}$-ideal terms;
\item $1\in F$ and $F$ is closed under the following ideal terms:
\begin{align*}
t_1(x,y_1,y_2) &= (y_1\ra (y_2\ra x))\ra x,\\
t_2(x,y) &= (y\rs x)\rs x,\\
t_3(y) &= y\ra 1.
\end{align*}
\end{enumerate}
\end{prop}

\begin{proof}
(1) $\Rightarrow$ (2). If $F\in\fil{A}$, then $F=[1]_\theta$ for some $\theta\in\conq{PBCI} A$, hence $F$ is closed under all ideal terms.
(2) $\Rightarrow$ (3). Trivial.
(3) $\Rightarrow$ (1). Suppose that $1\in F$ and that $F$ is closed under the three ideal terms. If $a,a\ra b\in F$, then $b=1\ra b=((a\ra b)\ra (a\ra b))\ra b = t_1 (b,a\ra b,a)\in F$ and also $a\ra 1 = t_3 (a)\in F$. 
Moreover, for any $a\in A$ and $b\in F$ we have $(b\ra a)\ra a=(1\ra (b\ra a))\ra a=t_1(a,1,b)\in F$ and $(b\rs a)\rs a=t_2(a,b)\in F$. Hence $F\in\fil A$.
\end{proof}

Instead of $t_1(x,y_1,y_2)$ and $t_2(x,y)$ we could take the ideal term
\[
w(x_1,x_2,y_1,y_2)=([(y_1\ra (y_2\ra x_1))\ra x_1]\rs x_2)\rs x_2.
\]

\begin{cor}
The quasivarieties $\var{PBCK}$ and $\var{PBCI}$ are relatively ideal determined.
\end{cor}

By \cite[Corollary 1.5]{GU}, congruence lattices of algebras in ideal determined varieties are arguesian, and hence modular. 
For relatively ideal determined quasivarieties we have the following analogue.
For the reader's convenience we give a proof which, however, is but a reformulation of the proof of \cite[Lemma 2.1]{Jo53} (also see \cite[Theorem IV.4.10]{gratzer}).
Let us recall that a lattice is said to be \emph{arguesian} if it satisfies the identity
\[
(x_1\vee y_1) \wedge (x_2\vee y_2) \wedge (x_3\vee y_3) \leq (x_1 \wedge (x_2\vee z)) \vee (y_1 \wedge (y_2\vee z)),
\]
where $z=z_{12} \wedge (z_{13} \vee z_{23})$ and $z_{ij}=(x_i\vee x_j) \wedge (y_i\vee y_j)$ for $i<j$. An arguesian lattice is a fortiori modular; see \cite[pp. 260--261]{gratzer}.

\begin{prop}\label{rid-rcm}
Let $\var Q$ be a relatively ideal determined quasivariety.
\begin{enumerate}[\indent\upshape (1)]
\item
For every $\alg A\in\var Q$, $\bconq QA$ is an arguesian lattice. Hence $\var Q$ is relatively congruence modular.
\item 
If there is a binary term $s(x,y)$ in the language of $\var Q$ such that $\var Q$ satisfies the identities 
$s(x,1) = x$ and $s(x,x) = 1 = s(1,x)$, 
then $\var Q$ is relatively congruence distributive.
\end{enumerate}
\end{prop}

\begin{proof}
Let $\alg A$ be a fixed algebra in a relatively ideal determined quasivariety $\var Q$ with a constant $1$.
First, we observe that for any $\phi, \psi\in \conq Q A$ and $a\in A$ we have 
\begin{equation}\label{spojeniidealu}
(a,1)\in\phi\vee\psi \quad\text{iff}\quad (a,1)\in\phi\circ\psi \quad\text{iff}\quad (a,1)\in\psi\circ\phi,
\end{equation}
whence the join of the $\var Q$-ideals $[1]_\phi$ and $[1]_\psi$ in the $\var Q$-ideal lattice $\bidq Q A$ is
$[1]_\phi \vee [1]_\psi = \{a\in A\mid (a,1)\in \phi\circ\psi\} = \{a\in A\mid (a,1)\in\psi\circ\phi\}$.
This is basically \cite[Lemma 1.4]{GU} rephrased for relatively ideal determined classes (only ideal determined classes are considered in \cite{GU}).

Since $[1]_{\phi}\vee [1]_{\psi} = [1]_{\phi\vee\psi}$ is the $\var Q$-ideal generated by $[1]_{\phi} \cup [1]_{\psi}$, we have $(a,1)\in\phi\vee\psi$ iff there is a $\var Q$-ideal term $t(x_1,\dots,x_m,y_1,\dots,y_n)$ in $y_1,\dots,y_n$ and elements $a_1,\dots,a_m\in A$, $b_1,\dots,b_n\in [1]_\phi\cup [1]_\psi$ such that
$a = t(a_1,\dots,a_m,b_1,\dots,b_n)$. For each $i = 1,\dots,n$, let $c_i = 1$ if $(b_i,1)\in\phi$, and $c_i = b_i$ otherwise. Then
\[
a=t(a_1,\dots,a_m,b_1,\dots,b_n) \equiv_\phi t(a_1,\dots,a_m,c_1,\dots,c_n)
\] 
since $(b_i,c_i)\in \phi$ for all $i$, and
\[
t(a_1,\dots,a_m,c_1,\dots,c_n) \equiv_\psi t(a_1,\dots,a_m,1,\dots,1) = 1
\]
since $(c_i,1)\in\psi$ for all $i$, whence $(a,1)\in \phi\circ\psi$. The opposite direction is clear as $\phi\circ\psi\subseteq\phi\vee\psi$. This completes the proof of \eqref{spojeniidealu}.

(1) To prove that $\bconq QA \cong \bidq QA$ is an arguesian lattice, we take $\phi_i,\psi_i\in\conq QA$, for $i=1,2,3$, and put 
\[
\theta = \theta_{12} \cap (\theta_{13} \vee \theta_{23}),
\]
where
\[
\theta_{ij}=(\phi_i \vee \phi_j) \cap (\psi_i \vee \psi_j)
\]
for $i<j$.

Suppose that $(a,1)\in (\phi_1 \vee \psi_1) \cap (\phi_2 \vee \psi_2) \cap (\phi_3 \vee \psi_3)$. 
By \eqref{spojeniidealu} there exist $b_i\in A$ such that $(a,b_i)\in\phi_i$ and $(b_i,1)\in\psi_i$ for $i=1,2,3$.
Obviously, $(b_i,b_j)\in (\phi_i\circ\phi_j) \cap (\psi_i\circ\psi_j)\subseteq \theta_{ij}$. Then $(b_1,b_2)\in \theta_{12} \cap (\theta_{13}\circ \theta_{23}) \subseteq \theta$. Moreover, $(a,b_1)\in \phi_1 \cap (\phi_2\circ\theta)$ since $(a,b_2)\in\phi_2$ and $(b_2,b_1)\in\theta$, and $(b_1,1)\in \psi_1\cap (\theta\circ\psi_2)$ since $(b_2,1)\in\psi_2$ and $(b_1,b_2)\in\theta$.
Hence $(a,1)\in (\phi_1 \cap (\phi_2\circ\theta)) \circ (\psi_1 \cap (\theta\circ\psi_2)) \subseteq (\phi_1 \cap (\phi_2\vee\theta))\vee (\psi_1 \cap (\psi_2\vee\theta))$.
Recalling \eqref{spojeniidealu} we see that the lattice $\bidq QA$ is arguesian.

(2) Let $\phi,\chi,\psi\in\conq{Q}{A}$ and $(a,1)\in \phi \cap (\chi\vee\psi)$. By \eqref{spojeniidealu} there exists $b\in A$ such that $(a,b)\in\chi$ and $(b,1)\in\psi$. 
Let $c=s(a,s(s(a,b),b))$. Since $\var{Q}$ satisfies the identities $s(x,1) = x$ and $s(x,x) = 1 = s(1,x)$, we have 
\[
c \equiv_\phi s(1,s(s(1,b),b)) = 1,\quad
c \equiv_\chi s(a,s(s(a,a),a)) = a,\quad
c \equiv_\psi s(a,s(s(a,1),1)) = 1.
\]
Hence $(a,c) \in \phi\cap\chi$ and $(c,1) \in \phi\cap\psi$, and so $(a,1)\in (\phi\cap\chi)\circ (\phi\cap\psi) \subseteq (\phi\cap\chi)\vee (\phi\cap\psi)$.
\end{proof}

The statement (2) is a corollary of \cite[Corollary 12.2.5]{BR08} which says that a relatively $1$-regular quasivariety $\var Q$ is relatively congruence distributive if there is a binary term $s(x,y)$ such that $\var Q$ satisfies the identities of (2). In the case of $\var{PBCK}$ it suffices to take $s(x,y)=y\ra x$. 

\begin{cor}
The relative congruence lattice of any pseudo-BCI-algebra {\upshape(}resp. pseudo-BCK-algebra{\upshape)} is arguesian {\upshape(}resp. distributive{\upshape)}. Thus $\var{PBCI}$ {\upshape(}resp. $\var{PBCK}${\upshape)} is relatively congruence modular {\upshape(}resp. distributive{\upshape)}.
\end{cor}

For an alternative proof of relative congruence distributivity of pseudo-BCK-algebras, see \cite[Corollary 2.1.12 and 2.2.9]{JK-hab} or \cite{HK}.

By \cite[Corollary 1.9]{GU}, a $1$-regular variety $\var V$ is ideal determined if and only if there exists a binary term $s(x,y)$ such that $\var V$ satisfies the identities 
\[
s(x,x) = 1 \quad\text{and}\quad s(x,1) = x;
\] 
cf. Proposition \ref{rid-rcm} (2). Such a term is called a \emph{subtractive term} for $\var V$.
Unfortunately, a relatively $1$-regular quasivariety having a subtractive term may not be relatively ideal determined. For example (cf. \cite[Example 7.7]{BR99} or \cite[Example 4.7.5]{BR08}), if $\var Q$ is the quasivariety of all torsion-free abelian groups, then the additive group of integers $\mathbb{Z}$ belongs to $\var Q$ and the only $\var Q$-congruences of $\mathbb{Z}$ are $\Delta$ and $\nabla$. But the $\var Q$-ideals of $\mathbb{Z}$ are precisely the subgroups of $\mathbb{Z}$, hence $\var Q$ is not relatively ideal determined, though it is relatively $0$-regular and satisfies the identities $x = x-0$ and $x-x = 0$, i.e., $s(x,y)=x-y$ is a subtractive term for $\var Q$.

It was proved in \cite[Theorem 12.2.6]{BR08} that a relatively $1$-regular quasivariety $\var Q$ is relatively congruence modular whenever it is \emph{$1$-conservative}, in the sense that $\var Q$ and the variety $\mathrm{HSP}(\var Q)$ generated by $\var Q$ satisfy the same quasi-identities of the form 
\[\textstyle
\bigwedge_{i=1}^{n} s_i(x_1,\dots,x_k) = 1 \quad\Rightarrow\quad t(x_1,\dots,x_k) = 1.
\]
Conservativeness has some notable consequences (see \cite[Corollary 12.1.2]{BR08}); for instance, relatively (finitely) subdirectly irreducible algebras in $\var Q$ are (finitely) subdirectly irreducible in the absolute sense.
That $\var{PBCK}$ has this property was proved in \cite[Corollary 3.6]{vA06}.

\begin{prop}
Let $\var Q$ be a relatively $1$-regular quasivariety. Then $\var Q$ is relatively ideal determined if and only if $\var Q$ is $1$-conservative and has a subtractive term.
\end{prop}

\begin{proof} 
This follows from \cite[Corollary 14.3.5]{BR08}; in fact, the statement is \cite[Corollary 14.3.5]{BR08} particularized to relatively $1$-regular quasivarieties. Conservativeness of $\var Q$ could be established directly, using ideal terms and \eqref{spojeniidealu}.

Also, it follows from \cite[Proposition 7.2]{BR99} that a relatively ideal determined quasivariety has a subtractive term. Here, we give a straightforward proof based on the satisfaction of \eqref{spojeniidealu}.
For any $\alg A\in\var Q$ and $\rho\subseteq A^2$, $\mathrm{Cg}_{\var Q}(\rho)$ denotes the $\var Q$-congruence of $\alg A$ generated by $\rho$.

Let $\alg F$ be the $\var Q$-free algebra $\alg{F}_{\var Q}(x,y)$; then $\alg F\in\var Q$.
Clearly, $(x,1) \in \mathrm{Cg}_{\var Q}(x,y) \circ \mathrm{Cg}_{\var Q}(y,1)$, and hence also $(x,1) \in \mathrm{Cg}_{\var Q}(y,1) \circ \mathrm{Cg}_{\var Q}(x,y)$ by \eqref{spojeniidealu}. Thus there exists $s(x,y)\in\alg F$ such that 
$(x,s(x,y)) \in \mathrm{Cg}_{\var Q}(y,1)$ and $(s(x,y),1) \in \mathrm{Cg}_{\var Q}(x,y)$.
Let $\alg A$ be an arbitrary algebra in $\var Q$. For any $a\in A$, the assignment $x\mapsto a$, $y\mapsto 1$ determines a unique homomorphism $\eta$ from $\alg F$ to $\alg A$. The kernel congruence of $\eta$ is a $\var Q$-congruence of $\alg F$, hence $\mathrm{Cg}_{\var Q}(y,1)\subseteq\ker\eta$. But then $a=\eta(x)=\eta(s(x,y))=s(\eta(x),\eta(y))=s(a,1)$.
This proves that $\var Q$ satisfies $x = s(x,1)$.
The argument for the identity $s(x,x)= 1$ is similar.
\end{proof}

\begin{cor}
The quasivarieties $\var{PBCK}$ and $\var{PBCI}$ are $1$-conservative.
\end{cor}

%%%%%%%%%%%%%%%%%%%%%%%%%%%%%%%%%%%%%%%%%%%%%%%%%%%%%%%%%%%%%%%%%%%%%%

In conclusion, we return to pseudo-BCI-algebras. Let $\alg A$ be a pseudo-BCI-algebra. 
Dymek \cite{D14} described the ``deductive system'' generated by $\emptyset\neq S\subseteq A$; our goal is to describe the prefilter $\mathrm{Pg}(S)$ generated by a given non-empty subset $S$, with an intent to prove that $\bfil A$ is a sublattice of $\bpfil A$.
Recall that $a_1\dots a_n\ra x$ and $a_1\dots a_n\rs x$ are shorthands for $a_1\ra (\dots \ra (a_n\ra x)\dots)$ and $a_n\rs (\dots \rs (a_1\rs x)\dots)$, respectively.

\begin{lem}\label{L6}
Let $\alg{A}$ be a pseudo-BCI-algebra. For any $\emptyset\neq S\subseteq A$ one has
\[
\mathrm{Pg}(S) = \{x\in A\mid a_1\dots a_n\ra x=1 \text{ for some } a_1,\dots,a_n\in S\cup (S\ra 1), n\in\N\}
\]
where $S\ra  1=\{x\ra 1\mid x\in S\}$.
\end{lem}

\begin{proof}
Let $M$ be the set above, i.e., $x\in M$ iff $\alpha\ra x=1$ (or equivalently, $\alpha\rs x=1$) for some $\alpha=a_1\dots a_n$ with $a_1,\dots,a_n\in S\cup (S\ra 1)$. 

(i) For any $a\in S$, we have $(a\ra 1)\ra (a\ra 1)=1$, so $1\in M$. 

(ii) Let $x,x\ra y\in M$, i.e., $\alpha\rs (x\ra y)=1$ and $\beta\rs x=1$ for some non-empty words $\alpha,\beta$ over $S\cup (S\ra 1)$. By Lemma \ref{L1}, $x\ra (\alpha\rs y)=\alpha\rs (x\ra y)=1$, so $x\leq \alpha\rs y$ which yields 
$1=\beta\rs x \leq \beta\rs (\alpha \rs y)=\alpha\beta\rs y$. Hence $\alpha\beta\rs y=1$ and $y\in M$. 

(iii) Let $x\in M$, i.e., $a_1\dots a_n\ra x=1$ for some $a_1,\dots,a_n\in S\cup (S\ra 1)$. Then, using Lemma \ref{L1} \eqref{vl11}, we have 
\begin{align*}
1 & =(a_1\dots a_n\ra x)\ra 1\\
& = [a_1\ra (\dots\ra (a_n\ra x)\dots)]\ra 1\\
& =(a_1\ra 1)\rs (\dots\rs [(a_n\ra1 )\rs (x\ra 1)]\dots).
\end{align*}
If $a_i\in S$, then $a_i\ra 1\in S\cup (S\ra 1)$. On the other hand, if $a_i\in S\ra 1$, then $a_i = b_i\ra 1$ for some $b_i \in S$, and since $b_i\leq (b_i\ra 1)\ra 1=a_i\ra 1$, when replacing $a_i\ra 1$ with $b_i$ we get
\begin{align*}
1 & = (a_1\ra 1)\rs (\dots\rs [(a_i\ra 1)\rs (\dots\rs [(a_n\ra 1)\rs (x\ra 1)]\dots)]\dots)\\
& \leq (a_1\ra 1)\rs (\dots\rs [b_i\rs (\dots\rs [(a_n\ra 1)\rs (x\ra 1)]\dots)]\dots)
\end{align*} 
by Lemma \ref{L1} \eqref{vl5}, \eqref{vl9}.
Repeating this procedure for each $i=1,\dots,n$ we find $c_1,\dots c_n\in S\cup (S\ra 1)$ such that $c_1\rs (\dots\rs (c_n\rs (x\ra 1))\dots) = 1$, thus $x\ra 1\in M$. 

We have proved that $M\in \pfil A$. Clearly, $S\subseteq M$. If $P\in\pfil A$ is such that $S\subseteq P$, then also $S\ra 1\subseteq P$ and it follows that $M\subseteq P$. 
Thus $M=\mathrm{Pg}(S)$.
\end{proof}

\begin{rem}
We know that $I_{\alg A}$ is a filter of $\alg A$; in fact, $I_{\alg A}$ is the kernel of the homomorphisms $\gamma$ and $\delta$ that are defined by \eqref{gamma} and \eqref{delta}, respectively. However, $G_{\alg A}$ is neither a filter nor a prefilter of $\alg A$ in general. 
Consequently, though $\alg{G}_{\alg A}$ is a subalgebra the prefilters of which are precisely the subgroups of the group $(G_{\alg A},\cdot,1)$, the prefilter of $\alg A$ generated by a subgroup $S\subseteq G_{\alg A}$ need not be contained in $G_{\alg A}$.
\end{rem}

\begin{cor}
For any pseudo-BCI-algebra $\alg A$, the filter lattice $\bfil A$ is a complete sublattice of the prefilter lattice $\bpfil A$.
\end{cor}

\begin{proof}
Let $\phi,\psi \in \conq{PBCI}{A}$. Let $\vee$ and $\sqcup$ denote the join in $\bfil A$ and in $\bpfil A$, respectively.
By \eqref{spojeniidealu}, if $a\in [1]_\phi \vee [1]_\psi$, then $(a,1)\in\phi\circ\psi$, so there exists $b\in A$ such that $(a,b)\in\phi$ and $(b,1)\in\psi$. Then $b\ra a\in [1]_\phi$, $b\in [1]_\psi$, and since $b\ra ((b\ra a)\rs a)=1$, it follows that $a\in \mathrm{Pg} ([1]_\phi \cup [1]_\psi) = [1]_\phi \sqcup [1]_\psi$. Thus $[1]_\phi \vee [1]_\psi\subseteq [1]_\phi \sqcup [1]_\psi$; the converse inclusion is evident, and hence $\bfil A$ is a sublattice of $\bpfil A$.
Compactness entails that it is actually a complete sublattice.
\end{proof}

%%%%%%%%%%%%%%%%%%%%%%%%%%%%%%%%%%%%%%%%%%%%%%%%%%%%%%%%%%%%%%%%%%%%%%

\section{The direct product of $\alg{I}_{\alg A}$ and $\alg{G}_{\alg A}$}
\label{sect-direktnisoucin}

In this section, we generalize another result of \cite{RvA}; namely, we characterize pseudo-BCI-algebras $\alg A$ that are isomorphic to the direct product $\alg I_{\alg A} \times \alg G_{\alg A}$. The necessary and sufficient condition given in \cite[Theorem 20]{RvA} is associativity of the operation which is defined by the same rule as the group multiplication on $G_{\alg A}$.
When $\alg{A}$ is a proper pseudo-BCI-algebra, then for $x,y\in G_{\alg A}$ we have 
\begin{align*}
x\cdot y &= (x\ra 1)\rs y=(x\ra 1)\rs ((y\rs 1)\ra 1)\\
&= (y\rs 1)\ra ((x\ra 1)\rs 1)=(y\rs 1)\ra x, 
\end{align*}
but the elements $(x\ra 1)\rs y$ and $(y\rs 1)\ra x$ may be distinct if $x,y\notin G_{\alg A}$. Therefore, for any $x,y\in A$ we define
\[
x\ci y=(x\ra 1)\rs y \quad\text{and}\quad  x\st y=(y\rs 1)\ra x.
\]
The operations $\ci$ and $\st$ need not coincide even when $\alg A$ is a BCI-algebra, but in BCI-algebras we have $x\ci y=y\st x$ for all $x,y\in A$, which fails to be true in pseudo-BCI-algebras.

\begin{lem}\label{L3}
In any pseudo-BCI-algebra $\alg{A}$, for all $x,y,z\in A$:
\begin{enumerate}[\indent\upshape (1)]
\item $1\ci x = x = x\st 1$;
\item $x\ci ( x\ra 1) = 1 = (x\rs 1)\st x$;
\item $x\ci (y\st z) = (x\ci y)\st z$;
\item $(x\ci y)\ci z \leq x\ci (y\ci z)$;
\item $x\st (y\st z) \leq (x\st y)\st z$.
\end{enumerate}
\end{lem}

\begin{proof}
(1) and (2) are trivial calculations.
Also (3) is obvious because by Lemma \ref{L1} (6) we have 
$x\ci (y\st z) = (x\ra1)\rs((z\rs1)\ra y) = (z\rs1)\ra((x\ra1)\rs y) = (x\ci y)\st z$.
For (4), we have 
\begin{align*}
x\ra 1 &\leq ((x\ra 1)\rs y)\ra y && \text{by Lemma \ref{L1} \eqref{vl3}}\\
&\leq (y\ra 1)\rs (((x\ra 1)\rs y)\ra 1) && \text{by Lemma \ref{L1} \eqref{vl2}}\\ 
&\leq [(((x\ra 1)\rs y)\ra 1)\rs z]\ra ((y\ra 1)\rs z) && \text{by Lemma \ref{L1} \eqref{vl2}}\\
&=((x\ci y)\ci z)\ra (y\ci z)
\end{align*}
which is equivalent to $(x\ci y)\ci z \leq (x\ra1)\rs (y\ci z) = x\ci (y\ci z)$ by Lemma \ref{L1} \eqref{vl7}.
The proof of (5) is analogous.
\end{proof}

\begin{rem}
If $x\ci y=x\st y$ for all $x,y\in A$, then $(x\ra 1)\rs 1=x\ci 1=x\st 1=x$ for every $x\in A$, and so $x\in G_{\alg A}$. Thus the operations $\ci$ and $\st$ coincide if and only if $G_{\alg A}=A$. 
The identity $x\st (y\ci z)=(x\st y)\ci z$ obtained by switching $\ci$ and $\st$ in (3) does not hold true in general; it actually entails that $\ci$ and $\st$ coincide because $x\st z=x\st (1\ci z)=(x\st 1)\ci z=x\ci z$ for all $x,z\in A$.
\end{rem}

\begin{rem}
The operations $\cd$ and $\ra$ (or $\st$ and $\rs$) do not satisfy the residuation law \eqref{rl} in general. It is easy to show that $x\cd y\leq z$ implies $x\leq y\ra z$ (and $x\st y\leq z$ implies $y\leq x\rs z$), but the converse implication holds only if $G_{\alg A}=A$.
Indeed, assuming that $x\cd y\leq z$ iff $x\leq y\ra z$ for all $x,y,z\in A$, we see that for any $x\in A$, $x=1\ra x$ entails $x\leq (x\ra 1)\rs 1=x\cd 1\leq x$, thus $x\in G_{\alg A}$.
\end{rem}

\begin{lem}\label{lemJ1}
For any pseudo-BCI-algebra $\alg{A}$, the following are equivalent:
\begin{enumerate}[\indent\upshape (1)]
\item
the operation $\ci$ is associative; 
\item
$(g\ci h)\ci x=g\ci (h\ci x)$ for all $g,h\in G_{\alg A}$ and $x\in A$;
\item
$g^{-1} \rs (g\rs x)=x$ for all $g\in G_{\alg A}$ and $x\in A$;
\item
the operation $\st$ is associative;
\item
$x\st (h\st g)=(x\st h)\st g$ for all $g,h\in G_{\alg A}$ and $x\in A$;
\item
$g^{-1} \ra (g\ra x)=x$ for all $g\in G_{\alg A}$ and $x\in A$.
\end{enumerate}
\end{lem}

\begin{proof}
(1) $\Rightarrow$ (2). Trivial.

(2) $\Rightarrow$ (3). Since $g\cdot g^{-1}=1$, we have $x = 1\ci x = (g\ci g^{-1})\ci x = g\ci (g^{-1}\ci x) = (g\ra 1)\rs ((g^{-1}\ra 1)\rs x) = g^{-1} \rs (g\rs x)$.

(3) $\Rightarrow$ (1). We first note that $x\ra 1=((x\ra 1)\rs y)\ra y$ for all $x,y\in A$, because $x\ra 1\leq ((x\ra 1)\rs y)\ra y$ and $x\ra 1\in G_{\alg A}$ is a maximal element. Then 
\begin{align}
\begin{split}\label{slbguign346}
x\ra (y\ra 1) &=x\ra (y\rs 1)=y\rs (x\ra 1)= y\rs [((x\ra 1)\rs y)\ra y]\\
&=((x\ra 1)\rs y)\ra (y\rs y)=((x\ra 1)\rs y)\ra 1\\
&=(x\ci y)\ra 1.
\end{split}
\end{align}
Further, since $y\ra 1\in G_{\alg A}$ and $(y\ra 1)^{-1}\rs ((y\ra 1)\rs z)=z$ by (3), we have
\begin{alignat*}{2}
x\ci (y\ci z) &= (x\ra 1)\rs ((y\ra 1)\rs z)\\
&\leq [(y\ra 1)^{-1}\rs (x\ra 1)]\rs [(y\ra 1)^{-1}\rs ((y\ra 1)\rs z)]  
& \hspace{3mm} & \text{by Lem. \ref{L1} (8)}\\
&=[x\ra ((y\ra 1)^{-1}\rs 1)]\rs z\\
&=(x\ra (y\ra 1))\rs z\\
&=((x\ci y)\ra 1)\rs z && \text{by \eqref{slbguign346}}\\
&=(x\ci y)\ci z.
\end{alignat*}
The converse inequality holds by Lemma \ref{L3} (4), and hence $x\ci (y\ci z)=(x\ci y)\ci z$.

The proof of (4) $\Leftrightarrow$ (5) $\Leftrightarrow$ (6) is parallel to the above proof of (1) $\Leftrightarrow$ (2) $\Leftrightarrow$ (3), hence it remains to show that, e.g., (3) $\Leftrightarrow$ (6).
For any $x,y,z\in A$, 
\[
y\leq (y\ra (x\ra z))\rs (x\ra z)=x\ra [(y\ra (x\ra z))\rs z]
\]
implies $x\leq y\rs [(y\ra (x\ra z))\rs z]$ (by Lemma \ref{L1} \eqref{vl3}, \eqref{vl6} and \eqref{vl7}), and hence 
\[
x\rs (y\rs [(y\ra (x\ra z))\rs z])=1.
\]
In particular, for any $g\in G_{\alg A}$ and $x\in A$, 
$g\rs (g^{-1}\rs [(g^{-1}\ra (g\ra x))\rs x])=1$. 
But if $\alg{A}$ fulfills (3), then we also have 
\[
g\rs (g^{-1}\rs [(g^{-1}\ra (g\ra x))\rs x])=(g^{-1}\ra (g\ra x))\rs x,
\] 
so $(g^{-1}\ra (g\ra x))\rs x=1$ which proves $g^{-1}\ra (g\ra x)\leq x$. 
The converse inequality is clear: $x=1\ra x\leq (g\ra 1)\ra (g\ra x)=g^{-1}\ra (g\ra x)$. Thus $\alg{A}$ satisfies (6) whenever it satisfies (3). To prove the converse implication, it suffices to switch the arrows.
\end{proof}

Theorem \ref{thm-direktnirozklad} below generalizes \cite[Theorem 20]{RvA} which states that a BCI-algebra $\alg{A}$ is isomorphic to the direct product $\alg{I}_{\alg A} \times \alg{G}_{\alg A}$ if and only if the operation $\st$ is associative.
Obviously, if $\alg{A}$ is a BCI-algebra, then $\alg{I}_{\alg A}$ is a BCK-algebra, $\alg{G}_{\alg A}=\alg{G}_{\alg A}^{\dagger}$ and $x\ci y=y\st x$ for all $x,y\in A$, so the group $(G_{\alg A},\cdot,1)$ is abelian.

First an easy lemma that will be used in proving the theorem:

\begin{lem}\label{L4}
Let $\alg{A}$ be a pseudo-BCI-algebra. Then for all $x\in A$ and $g, h\in G_{\alg A}$:
\begin{enumerate}[\indent\upshape (1)]
\item $x\ra g = (g\ra x)\ra 1$, $x\rs g = (g\rs x)\ra 1$;
\item $g\ra x = x \st g^{-1}$, $g \rs x = g^{-1}\ci x$;
\item $(x\ra g)\ra 1=x\ci g^{-1}$, $(x\rs g)\ra 1=g^{-1}\st x$.
\end{enumerate}
\end{lem}

\begin{proof}
We have $(g\ra x)\ra 1 = (g\ra 1) \rs (x\ra 1) = x \ra ((g\ra 1)\rs 1) = x\ra g$, and similarly,  $(g\rs x)\ra 1 = x\rs g$.
Clearly, $x\st g^{-1} = (g^{-1}\rs 1)\ra x = g\ra x$, $g^{-1}\ci x =(g^{-1}\ra 1)\rs x = g\rs x$,
$x\ci g^{-1}=(x\ra 1)\rs (g\ra 1)=(x\ra g)\ra 1$ and $g^{-1}\st x=(x\rs 1)\ra (g\rs 1)=(x\rs g)\ra 1$.
\end{proof}

\begin{thm}\label{thm-direktnirozklad}
Let $\alg A$ be a pseudo-BCI-algebra. The following statements are equivalent: 
\begin{enumerate}[\indent\upshape (1)]
\item $\alg A \cong \alg{I_A} \times \alg{G_A} \cong \alg{I_A} \times \alg{G}_{\alg{A}}^{\dagger}$;
\item $G_{\alg A}$ is a filter of $\alg A$;
\item $\alg A$ satisfies the equivalent conditions of Lemma \ref{lemJ1} and 
\begin{equation}\label{123456}
g\ra x=g\rs x
\end{equation}
for all $g\in G_{\alg A}$, $x\in I_{\alg A}$.
\end{enumerate}
\end{thm}

\begin{rem}
Dymek \cite{D15} recently proved that $\alg A \cong \alg{I_A} \times \alg{G_A}$ if and only if $G_{\alg A}$ is a ``compatible deductive system'' of $\alg A$ (i.e., $G_{\alg A}$ satisfies the conditions (i), (ii) and (iv)), which is actually equivalent to $G_{\alg A}$ being a filter because $G_{\alg A}$ always satisfies the condition (iii). Our proof below is different from that in \cite{D15}.
\end{rem}

\begin{proof}[Proof of the theorem]
(1) $\Rightarrow$ (2). 
Let $\alg B$ be the direct product $\alg{I_A}\times\alg{G_A}$.
We have $\alg{I_A} \times \alg{G_A} \cong \alg{I_A} \times \alg{G}_{\alg{A}}^{\dagger}$. 
In view of Lemma \ref{lem-grcast},
the group part of $\alg B$ is $G_{\alg B}=\{(1,g)\mid g\in G_{\alg A}\}$ and any isomorphism $\eta$ between $\alg A$ and $\alg B$ maps $G_{\alg A}$ onto $G_{\alg B}$. Hence, if $\pi$ is the projection of $\alg B$ onto $\alg{I_A}$, then $G_{\alg A}$ is the kernel of the composite homomorphism $\pi\eta$ from $\alg A$ onto $\alg{I_A}$.
Thus $G_{\alg A}\in\fil A$.

(2) $\Rightarrow$ (3). 
Suppose that $G_{\alg A}$ is a filter of $\alg A$. Let $g\in G_{\alg A}$.
First, we show that $\alg A$ satisfies the condition (6) of Lemma \ref{lemJ1}.
For any $x\in A$ we have $x=1\ra x\leq (g\ra 1)\ra (g\ra x)=g^{-1}\ra (g\ra x)$, whence $1=x\rs x\geq (g^{-1}\ra (g\ra x))\rs x$. 
On the other hand, recalling Lemma \ref{L1} \eqref{vl3}, 
we get $g^{-1}\leq (g^{-1}\ra (g\ra x))\rs (g\ra x)=g\ra [(g^{-1}\ra (g\ra x))\rs x]$;
since $g^{-1}$ is a maximal element, we actually have 
$g^{-1}=g\ra [(g^{-1}\ra (g\ra x))\rs x]$.
Now, since both $g$ and $g^{-1}$ belong to the filter $G_{\alg A}$, it follows that $(g^{-1}\ra (g\ra x))\rs x\in G_{\alg A}$. But we have seen before that this element is $\leq 1$, and so $(g^{-1}\ra (g\ra x))\rs x=1$. Thus $g^{-1}\ra (g\ra x)\leq x$.

To verify \eqref{123456}, suppose that $x\in I_{\alg A}$. Then $(g\ra x)\ra 1=(g\ra 1)\rs (x\ra 1)=g^{-1}\rs 1=g$, and hence $x\leq 1$ implies $(g\ra x)\ra x\leq (g\ra x)\ra 1=g$. Since the element $g$ is maximal, we have $g=(g\ra x)\rs x\in G_{\alg A}$, whence $(g\ra x)\ra x\in G_{\alg A}$. However, $(g\ra x)\ra x\leq g$, and so $(g\ra x)\ra x=g$ which yields $g\ra x\leq ((g\ra x)\ra x)\rs x=g\rs x$. The proof of the inequality $g\rs x\leq g\ra x$ is parallel. 
Thus $g\ra x=g\rs x$ for all $g\in G_{\alg A}$ and $x\in I_{\alg A}$.

(3) $\Rightarrow$ (1). 
We prove the map 
\[
\eta\colon (a,g)\in I_{\alg A}\times G_{\alg A} \mapsto g\ra a\in A
\]
is an isomorphism from $\alg{I}_{\alg A}\times\alg{G}_{\alg A}^{\dagger}$ onto $\alg{A}$. 
For any $g,h\in G_{\alg A}$ and $a,b\in I_{\alg A}$ we have 
\begin{align}
\begin{split}\label{6546s4h4gfh64}
(g\rs h)\rs b &= (g\rs h)\ra b = b\st (g\rs h)^{-1} = b\st (h^{-1}\st g)= (b\st h^{-1})\st g\\
&= (h\ra b)\st g = g^{-1}\ra (h\ra b)
\end{split}
\end{align}
by \eqref{123456}, Lemma \ref{L4} and by associativity of $\st$. Then, since $g^{-1}\ra (g\ra a)=a$ by Lemma \ref{lemJ1} (6),
\begin{align*}
(g\rs h)\ra (a\ra b) &= (g\rs h)\rs (a\ra b) = a\ra ((g\rs h)\rs b)\\
&= a\ra (g^{-1}\ra (h\ra b)) = (g^{-1}\ra (g\ra a))\ra (g^{-1}\ra (h\ra b))\\ 
&\geq (g\ra a)\ra (h\ra b)
\end{align*}
by Lemma \ref{L1} (8).
On the other hand, again using Lemma \ref{lemJ1} (6), Lemma \ref{L1} \eqref{vl8} and \eqref{6546s4h4gfh64}, 
\begin{align*}
(g\ra a)\ra (h\ra b) &= (g\ra a)\ra [g\ra (g^{-1}\ra (h\ra b))]\\
&\geq a\ra (g^{-1}\ra (h\ra b)) = a\ra ((g\rs h)\rs b)\\
&= (g\rs h)\rs (a\ra b) = (g\rs h)\ra (a\ra b).
\end{align*}
Thus, for all $g,h\in G_{\alg A}$ and $a,b\in I_{\alg A}$ we have
\[
(g\rs h)\ra (a\ra b)=(g\ra a)\ra (h\ra b).
\] 
We can analogously show that 
\[
(g\ra h)\rs (a\rs b)=(g\rs a)\rs (h\rs b). 
\]
Therefore, since $I_{\alg A}\times G_{\alg A}$ is regarded as the universe of the pseudo-BCI-algebra $\alg{I}_{\alg A}\times\alg{G}_{\alg A}^{\dagger}$ (not of $\alg{I}_{\alg A}\times\alg{G}_{\alg A}$), 
we get 
$\eta ((a,g)\ra (b,h)) = \eta(a\ra b,g\rs h) = (g\rs h)\ra (a\ra b) = (g\ra a)\ra (h\ra b) = \eta(a,g)\ra\eta(b,h)$
and 
$\eta ((a,g)\rs (b,h)) = \eta(a\rs b,g\ra h) = (g\ra h)\rs (a\rs b) = (g\rs a)\rs (h\rs b) = \eta(a,g)\rs \eta(b,h)$.
Thus $\eta$ is indeed a homomorphism from $\alg{I}_{\alg A}\times\alg{G}_{\alg A}^{\dagger}$ onto $\alg{A}$.

It remains to show that $\eta$ is a bijection.
Surjectivity:
By Lemma \ref{lem-intcast}, $(x\ra 1)^{-1}\ra x\in I_{\alg A}$ for every $x\in A$. Since $x\ra 1\in G_{\alg A}$, by Lemma \ref{lemJ1} (6) we obtain 
\[
x = (x\ra 1)\ra ((x\ra 1)^{-1}\ra x) = \eta ((x\ra 1)^{-1}\ra x,x\ra 1),
\]
hence the map $\eta$ is onto.
Injectivity:
Suppose that $\eta(a,g) = 1$, i.e. $g\ra a = 1$. Then $g\leq a$ which entails $g = a$ as $g$ is a maximal element. Also $a\leq 1$, and hence $g = a = 1$. Thus $\eta^{-1}(1)=\{(1,1)\}$. Since $\eta$ is a homomorphism whose kernel is trivial, we conclude that $\eta$ is injective.
\end{proof}

The homomorphism $\delta$ defined by \eqref{delta} is a retraction of $\alg A$ onto $\alg{G}_{\alg A}$.
By Lemma \ref{lem-grcast} we know that $x\in G_{\alg A}$ iff $x=\delta(x)$. Using this fact and recalling Proposition \ref{prop-idealterms}, we obtain that $G_{\alg A}$ is a filter of $\alg A$ if and only if $\alg A$ satisfies the identities
\[
\delta(t_1(x,y_1,y_2)) = t_1(x,\delta(y_1),\delta(y_2))
\quad\text{and}\quad
\delta(t_2(x,y)) = t_2(x,\delta(y)),
\]
where $t_1(x,y_1,y_2)$ and $t_2(x,y)$ are the ideal terms of Proposition \ref{prop-idealterms}. Analogously, by Lemma \ref{lem-intcast} we have $\delta(x)\ra x\in I_{\alg A}$ for every $x\in A$, whence $x\in I_{\alg A}$ iff $x=\delta(x)\ra x$, and hence the condition \eqref{123456} can be captured by the identity 
\[
\delta(y)\ra (\delta(x)\ra x) = \delta(y)\rs (\delta(x)\ra x).
\]
Therefore:

\begin{cor}
The class of pseudo-BCI-algebras $\alg A$ isomorphic to $\alg{I}_{\alg A}\times\alg{G}_{\alg A}$ is a relative subvariety of the quasivariety $\var{PBCI}$.
\end{cor}

The condition \eqref{123456}, which is obviously redundant in the case of BCI-algebras, cannot be omitted. 
The following example (found by Prover9-Mace4) shows that there exist pseudo-BCI-algebras satisfying the conditions (1)--(6) of Lemma \ref{lemJ1} but not \eqref{123456}.

\begin{exa}
Let $\alg A$ be the pseudo-BCI-algebra defined by the tables

\begin{center}
\begin{tabular}{c|cccccc}
$\ra$ & $a$ & $b$ & $x$ & $y$ & $g$ & $1$\\ \hline
$a$    & $1$ & $b$ & $g$ & $y$ & $g$ & $1$\\
$b$   & $a$ & $1$ & $x$ & $g$ & $g$ & $1$\\
$x$   & $g$ & $x$ & $1$ & $a$ & $1$ & $g$\\
$y$   & $y$ & $g$ & $b$ & $1$ & $1$ & $g$\\
$g$   & $y$ & $x$ & $b$ & $a$ & $1$ & $g$\\
$1$   & $a$ & $b$ & $x$ & $y$ & $g$ & $1$
\end{tabular}
\qquad
\begin{tabular}{c|cccccc}
$\rs$ & $a$ & $b$ & $x$ & $y$ & $g$ & $1$\\ \hline
$a$    & $1$ & $b$ & $x$ & $g$ & $g$ & $1$\\
$b$   & $a$ & $1$ & $g$ & $y$ & $g$ & $1$\\
$x$   & $x$ & $g$ & $1$ & $b$ & $1$ & $g$\\
$y$   & $g$ & $y$ & $a$ & $1$ & $1$ & $g$\\
$g$   & $x$ & $y$ & $a$ & $b$ & $1$ & $g$\\
$1$   & $a$ & $b$ & $x$ & $y$ & $g$ & $1$
\end{tabular}
\end{center}
Then $I_{\alg A}=\{a,b,1\}$, $G_{\alg A}=\{g,1\}$ and $\alg A$ satisfies the equivalent conditions of Lemma \ref{lemJ1}, but, e.g., $g\ra a=y$ while $g\rs a=x$. Hence $\alg A$ is not isomorphic to $\alg{I_A}\times\alg{G_A}$, which is evident also from the fact that $\alg{I}_{\alg A}$ and $\alg{G}_{\alg A}$ are BCI-algebras, but $\alg A$ is a proper pseudo-BCI-algebra.
\end{exa}

%%%%%%%%%%%%%%%%%%%%%%%%%%%%%%%%%%%%%%%%%%%%%%%%%%%%%%%%%%%%%%%%%%%%%%

\section*{Acknowledgements}

The authors would like to thank the referees for their comments and suggestions.

Both P. E. and J. K. were supported by the Palack\'{y} University project ``Mathematical structures'' IGA PrF 2015010. 
J. K. was also supported by the bilateral project ``New perspectives on residuated posets'' of the Austrian Science Fund (FWF): project I 1923-N25, and the Czech Science Foundation (GA\v{C}R): project 15-34697L.

\nocite{*}
\bibliographystyle{amsplain}
\bibliography{references}

\providecommand{\bysame}{\leavevmode\hbox to3em{\hrulefill}\thinspace}
\providecommand{\MR}{\relax\ifhmode\unskip\space\fi MR }
% \MRhref is called by the amsart/book/proc definition of \MR.
\providecommand{\MRhref}[2]{%
  \href{http://www.ams.org/mathscinet-getitem?mr=#1}{#2}
}
\providecommand{\href}[2]{#2}
\begin{thebibliography}{10}

\bibitem{BP89}
W.~J. Blok and D.~Pigozzi, \emph{Algebraizable logics}, Memoirs of the AMS, no.
  396, AMS, Providence, Rhode Island, USA, 1989.

\bibitem{BR99}
W.~J. Blok and J.~G. Raftery, \emph{Ideals in quasivarieties of algebras},
  Models, Algebras and Proofs (X.~Caicedo and C.~H. Montenegro, eds.), Lecture
  Notes in Pure and Applied Mathematics, vol. 203, Marcel Dekker, New York,
  1999, pp.~167--186.

\bibitem{BR08}
\bysame, \emph{Assertionally equivalent quasivarieties}, Int. J. Algebra
  Comput. \textbf{18} (2008), 589--681.

\bibitem{Bo}
B.~Bosbach, \emph{Concerning cone algebras}, Algebra Univers. \textbf{15}
  (1982), 58--66.

\bibitem{BKLT}
M.~Botur, J.~K\"{u}hr, L.~Liu, and C.~Tsinakis, \emph{The {C}onrad program:
  {F}rom $\ell$-groups to algebras of logic}, J. Algebra, to appear.

\bibitem{CEL}
I.~Chajda, G.~Eigenthaler, and H.~L\"{a}nger, \emph{Congruence classes in
  universal algebra}, Research and Exposition in Mathematics, vol.~32,
  Heldermann Verlag, Lemgo, 2003.

\bibitem{ciungu}
L.~C. Ciungu, \emph{Non-commutative multiple-valued logic algebras}, Springer
  Monographs in Mathematics, Springer, 2014.

\bibitem{DNGI-bl}
A.~{D}i Nola, G.~Georgescu, and A.~Iorgulescu, \emph{Pseudo-{BL} algebras: Part
  {I}}, Mult.-Valued Log. \textbf{8} (2002), 673--714.

\bibitem{DJ}
W.~A. Dudek and Y.~B. Jun, \emph{Pseudo-{BCI} algebras}, East Asian Math. J.
  \textbf{24} (2008), 187--190.

\bibitem{DK}
A.~Dvure\v{c}enskij and J.~K\"{u}hr, \emph{On the structure of linearly ordered
  pseudo-{BCK}-algebras}, Arch. Math. Logic \textbf{48} (2009), 771--791.

\bibitem{D11}
G.~Dymek, \emph{On two classes of pseudo-{BCI}-algebras}, Discuss. Math. Gen.
  Algebra Appl. \textbf{31} (2011), 217--230.

\bibitem{D12a}
\bysame, \emph{Atoms and ideals of pseudo-{BCI}-algebras}, Comment. Math.
  \textbf{52} (2012), 73--90.

\bibitem{D12b}
\bysame, \emph{{p}-semisimple pseudo-{BCI}-algebras}, J. Mult.-Val. Log. Soft
  Comput. \textbf{19} (2012), 461--474.

\bibitem{D14}
\bysame, \emph{On compatible deductive systems of pseudo-{BCI}-algebras}, J.
  Mult.-Val. Log. Soft Comput. \textbf{22} (2014), 167--187.

\bibitem{D15}
\bysame, \emph{On a periodic part of pseudo-{BCI}-algebras}, Discuss. Math.
  Gen. Algebra Appl. (2015).

\bibitem{DKD13}
G.~Dymek and A.~Kozanecka-Dymek, \emph{Pseudo-{BCI}-logic}, Bull. Sect. Logic
  \textbf{42} (2013), 33--41.

\bibitem{Fl}
I.~Fleischer, \emph{Every {BCK}-algebra is a set of residuables in an integral
  pomonoid}, J. Algebra \textbf{119} (1988), 360--365.

\bibitem{GJKO}
N.~Galatos, P.~Jipsen, T.~Kowalski, and H.~Ono, \emph{Residuated lattices: An
  algebraic glimpse at substructural logics}, Studies in Logic and the
  Foundations of Mathematics, vol. 151, Elsevier, Amsterdam, 2007.

\bibitem{GI}
G.~Georgescu and A.~Iorgulescu, \emph{Pseudo-{BCK} algebras: An extension of
  {BCK} algebras}, Combinatorics, Computability and Logic (C.~S. Calude, M.~J.
  Dinneen, and S.~Sburlan, eds.), Springer, London, 2001, pp.~97--114.

\bibitem{GI-mv}
\bysame, \emph{Pseudo-{MV} algebras}, Mult.-Valued Log. \textbf{6} (2001),
  95--135.

\bibitem{GIP-ph}
G.~Georgescu, A.~Iorgulescu, and V.~Preoteasa, \emph{Pseudo-hoops}, J.
  Mult.-Val. Log. Soft Comput. \textbf{11} (2005), 153--184.

\bibitem{gorbunov}
V.~A. Gorbunov, \emph{Algebraic theory of quasivarieties}, Siberian School of
  Algebra and Logic, Plenum, New York, 1998.

\bibitem{gratzer}
G.~Gr\"{a}tzer, \emph{General lattice theory}, 2nd ed., Birkh\"{a}ser, Basel,
  2003.

\bibitem{GU}
H.~P. Gumm and A.~Ursini, \emph{Ideals in universal algebras}, Algebra Univers.
  \textbf{19} (1984), 45--54.

\bibitem{HK}
R.~Hala\v{s} and J.~K\"{u}hr, \emph{Deductive systems and annihilators of
  pseudo {BCK}-algebras}, Italian J. Pure Appl. Math. \textbf{25} (2009),
  83--94.

\bibitem{iorgulescu}
A.~Iorgulescu, \emph{Algebras of logic as {BCK}-algebras}, Editura ASE,
  Bucharest, 2008.

\bibitem{I66}
K.~Is\'{e}ki, \emph{An algebra related with a propositional calculus}, Proc.
  Japan Acad. \textbf{42} (1966), 26--29.

\bibitem{Jo53}
B.~J\'{o}nsson, \emph{On the representation of lattices}, Math. Scand.
  \textbf{1} (1953), 193--206.

\bibitem{JKN}
Y.~B. Jun, H.~S. Kim, and J.~Neggers, \emph{On pseudo-{BCI} ideals of
  pseudo-{BCI} algebras}, Mat. Vesnik \textbf{58} (2006), 39--46.

\bibitem{KS}
Y.~H. Kim and K.~S. So, \emph{On minimality in pseudo-{BCI}-algebras}, Commun.
  Korean Math. Soc. \textbf{27} (2012), 7--13.

\bibitem{Ko}
Y.~Komori, \emph{The class of {BCC}-algebras is not a variety}, Math. Japon.
  \textbf{29} (1984), 391--394.

\bibitem{JK-csmj05}
J.~K\"{u}hr, \emph{Ideals of noncommutative {DR$\ell$}-monoids}, Czechoslovak
  Math. J. \textbf{55} (2005), 97--111.

\bibitem{JK-cga}
\bysame, \emph{Pseudo {BCK}-algebras and residuated lattices}, Contributions to
  General Algebra (I.~Chajda, G.~Dorfer, G.~Eigenthaler, R.~Hala\v{s},
  W.~M\"{u}ller, and R.~P\"{o}schel, eds.), vol.~16, Johannes Heyn, Klagenfurt,
  2005, pp.~139--144.

\bibitem{JK-hab}
\bysame, \emph{Pseudo-{BCK}-algebras and related structures}, Habilitation
  Thesis, Palack\'{y} University, Olomouc, 2007.

\bibitem{JK-ja07}
\bysame, \emph{Representable pseudo-{BCK}-algebras and integral residuated
  lattices}, J. Algebra \textbf{317} (2007), 354--364.

\bibitem{JK-jmvlsc10}
\bysame, \emph{Boolean and central elements and {C}antor-{B}ernstein theorem in
  bounded pseudo-{BCK}-algebras}, J. Mult.-Val. Log. Soft Comput. \textbf{16}
  (2010), 387--404.

\bibitem{LP}
K.~J. Lee and C.~H. Park, \emph{Some ideals of pseudo {BCI}-algebras}, J. Appl.
  Math. Informatics \textbf{27} (2009), 217--231.

\bibitem{OK}
H.~Ono and Y.~Komori, \emph{Logics without the contraction rule}, J. Symb.
  Logic. \textbf{50} (1985), 169--201.

\bibitem{Pa}
M.~Palasi\'{n}ski, \emph{An embedding theorem for {BCK}-algebras}, Math. Sem.
  Notes Kobe Univ. \textbf{10} (1982), 749--751.

\bibitem{Ra}
J.~Rach\r{u}nek, \emph{A non-commutative generalization of {MV}-algebras},
  Czechoslovak Math. J. \textbf{52} (2002), 255--273.

\bibitem{RvA}
J.~G. Raftery and C.~J. van Alten, \emph{Residuation in commutative ordered
  monoids with minimal zero}, Reports on Math. Logic \textbf{34} (2000),
  23--57.

\bibitem{vA02}
C.~J. van Alten, \emph{Representable biresiduated lattices}, J. Algebra
  \textbf{247} (2002), 672--691.

\bibitem{vA06}
\bysame, \emph{On varieties of biresiduation algebras}, Studia Logica
  \textbf{83} (2006), 425--445.

\bibitem{bci-book}
H.~Yisheng, \emph{{BCI}-algebra}, Science Press, Beijing, 2006.

\bibitem{Z10b}
X.~Zhang, \emph{Pseudo-{BCI} filters and subalgebras of pseudo-{BCI} algebras},
  2010 IEEE Int. Conf. on Progress in Informatics and Computing (Shanghai),
  December 2010, pp.~140--144.

\bibitem{Z10a}
\bysame, \emph{Pseudo-{BCK} part and anti-grouped part of pseudo-{BCI}
  algebras}, 2010 IEEE Int. Conf. on Progress in Informatics and Computing
  (Shanghai), December 2010, pp.~127--131.

\bibitem{Z13}
\bysame, \emph{{BCC}-algebras and residuated partially ordered groupoids},
  Math. Slovaca \textbf{63} (2013), 397--410.

\bibitem{ZJ14}
X.~Zhang and Y.~B. Jun, \emph{Anti-grouped pseudo-{BCI} algebras anti-grouped
  filters}, Fuzzy Syst. Math. \textbf{28} (2014), 21--33.

\bibitem{ZL}
X.~Zhang and W.~H. Li, \emph{On pseudo-{BL} algebras and {BCC}-algebras}, Soft
  Comput. \textbf{10} (2006), 941--952.

\end{thebibliography}
\end{document}